\documentclass[12pt]{article}

\usepackage[a4paper,
textwidth = 16cm,
textheight = 22.5cm,
top = 3cm]{geometry}

\usepackage[colorlinks,citecolor=blue,urlcolor=blue,breaklinks]{hyperref}
\usepackage{xcolor,latexsym,amsmath,amssymb,amsfonts,amsthm,bm,bbm,enumitem,xcolor,soul,graphicx, dsfont,enumitem,threeparttable,caption,subcaption, comment, booktabs}

\graphicspath{{./fig/}}
\usepackage[normalem]{ulem}
\usepackage[ruled, vlined, linesnumbered]{algorithm2e}
\usepackage{multirow}
\usepackage[round]{natbib}

\usepackage[utf8]{inputenc}   % Or \usepackage{fontspec} for XeLaTeX
\usepackage[T1]{fontenc}      % Font encoding
\usepackage{lmodern}      % lmodern font, correctly copyable characters in pdf

\SetKwInput{KwSet}{Set}

\newtheorem{thm}{Theorem}
\newtheorem{prop}[thm]{Proposition}
\newtheorem{lemma}[thm]{Lemma}
\newtheorem{cor}[thm]{Corollary}

\DeclareMathOperator*{\argmax}{arg\,max}

\DeclareMathOperator*{\supp}{supp}
\DeclareMathOperator{\diag}{diag}
\DeclareMathOperator{\RSS}{RSS}

\newcommand\smallO{
  \mathchoice
    {{\scriptstyle\mathcal{O}}}% \displaystyle
    {{\scriptstyle\mathcal{O}}}% \textstyle
    {{\scriptscriptstyle\mathcal{O}}}% \scriptstyle
    {\scalebox{.5}{$\scriptscriptstyle\mathcal{O}$}}%\scriptscriptstyle
  }

\title{Two-sample testing of high-dimensional linear regression coefficients via complementary sketching}
\author{Fengnan Gao\footnote{Fudan University and Shanghai Center for Mathematical Sciences. Email: fngao@fudan.edu.cn} \; and\; Tengyao Wang\footnote{London School of Economics and Political Science; University College London. Email: t.wang59@lse.ac.uk}}
\date{(\today)}

\begin{document}
\maketitle
\begin{comment}
% text abstract
We introduce a new method for two-sample testing of high-dimensional linear regression coefficients without assuming that those coefficients are individually estimable. The procedure works by first projecting the matrices of covariates and response vectors along directions that are complementary in sign in a subset of the coordinates, a process which we call 'complementary sketching'. The resulting projected covariates and responses are aggregated to form two test statistics, which are shown to have essentially optimal asymptotic power under a Gaussian design when the difference between the two regression coefficients is sparse and dense respectively. Simulations confirm that our methods perform well in a broad class of settings and an application to a large single-cell RNA sequencing dataset demonstrates its utility in the real world.
\end{comment}

\begin{abstract}
%!TEX root = ../compsket.tex

We introduce a new method for two-sample testing of high-dimensional linear regression coefficients without assuming that those coefficients are individually estimable. The procedure works by first projecting the matrices of covariates and response vectors along directions that are complementary in sign in a subset of the coordinates, a process which we call `complementary sketching'. The resulting projected covariates and responses are aggregated to form two test statistics, which are shown to have essentially optimal asymptotic power under a Gaussian design when the difference between the two regression coefficients is sparse and dense respectively. Simulations confirm that our methods perform well in a broad class of settings and an application to a large single-cell RNA sequencing dataset demonstrates its utility in the real world.

\end{abstract}

%!TEX root = ../compsket.tex

\section{Introduction}
\label{sec:intro}
%\subsection{Background and related works}
Two-sample testing problems are commonplace in statistical applications across different scientific fields, wherever researchers want to compare observations from different samples.
In its most basic form, a two-sample Gaussian mean testing problem is formulated as follows: upon observing two samples $X_1,\dots X_{n_1} \stackrel{\mathrm{iid}}{\sim} N(\mu_1, \sigma^2) $ and $Y_1, \dots, Y_{n_2} \stackrel{\mathrm{iid}}{\sim} N(\mu_2, \sigma^2)$, we wish to test
\begin{equation}
    \label{Eqn:classical-two-sample}
    H_0: \mu_1 = \mu_2 \quad \text{versus} \quad H_1: \mu_1 \neq \mu_2.
\end{equation}
This leads to the introduction of the famous two-sample Student's $t$-test.
In a slightly more involved form in the parametric setting, we observe $X_1,\dots, X_{n_1} \stackrel{\mathrm{iid}}{\sim} F_{\theta_1, \gamma_1}$ and $Y_1, \dots, Y_{n_2} \stackrel{\mathrm{iid}}{\sim} F_{\theta_2, \gamma_2}$ and would like to test
\(
    H_0: \theta_1 = \theta_2 \text{ versus } H_1:  \theta_1 \neq \theta_2,
\)
where $\gamma_1$ and $\gamma_2$ are nuisance parameters.

Linear regression models have been one of the staples of statistics.
A two-sample testing problem in linear regression arises in the following classical setting:  fix $p \ll \min\{n_1, n_2\}$,
we observe an $n_1$-dimensional response vector $Y_1$ with an associated design matrix $X_1 \in \mathbb{R}^{n_1 \times p}$ in the first sample, and an $n_2$-dimensional response $Y_2$ with design matrix $X_2 \in \mathbb{R}^{n_2\times p}$ in the second sample. We assume in both samples the responses are generated from standard linear models
\begin{equation}
\begin{cases}
Y_1 =  X_1 \beta_1 + \epsilon_1, \\
Y_2 = X_2  \beta_2 + \epsilon_2,
\end{cases}
\label{Eq:Model}
\end{equation}
for some unknown regression coefficients $ \beta_1, \beta_2 \in \mathbb{R}^p$ and independent homoscedastic noise vectors $\epsilon_1\mid (X_1,X_2) \sim N_{n_1}(0, \sigma^2 I_{n_1})$ and $\epsilon_2 \mid (X_1,X_2)\sim N_{n_2}(0, \sigma^2 I_{n_2})$.
The purpose is to test
\(
    H_0: \beta_1 = \beta_2 \text{ versus }  H_1: \beta_1 \neq \beta_2.
\)
Suppose that $\hat\beta$ is the least square estimate of $\beta = \beta_1 = \beta_2$ under the null hypothesis and $\hat\beta_1$, $\hat\beta_2$ are the least square estimates of $\beta_1$ and $\beta_2$ respectively under the alternative hypothesis.
Define the residual sum of squares as
\begin{equation}
    \begin{aligned}
        \RSS_1 & = \| Y_1 - X_1 \hat\beta_1\|_2^2 + \| Y_2 - X_2 \hat\beta_2\|_2^2,\\
        \RSS_0 & = \| Y_1 - X_1\hat\beta\|_2^2 + \| Y_2 - X_2 \hat\beta\|_2^2.
    \end{aligned}
    \label{Eqn:rss-def}
\end{equation}
The classical generalized likelihood ratio test \citep{chow1960tests} compares the $F$-statistic
\begin{equation}
\label{Eq:LRT}
    %F = \frac{\| Y_1 - X_1\hat\beta\|^2 + \| Y_2 - X_2 \hat\beta\|^2 - \| Y_1 - X_1 \hat\beta_1\|^2 -\| Y_2 - X_2 \hat\beta_2\|^2 }{ \| Y_1 - X_1\hat\beta\|^2 + \| Y_2 - X_2 \hat\beta\|^2}\cdot \frac{n_1 + n_2 - 2p}{p} \sim F_{p, n_1+n_2-2p}
    F = \frac{(\RSS_0 - \RSS_1)/p }{ \RSS_1/(n_1+n_2-2p)} \sim F_{p,\, n_1+n_2-2p}
\end{equation}
against upper quantiles of the $F_{p, \, n_1+n_2-2p}$ distribution. It is well-known that in the classical asymptotic regime where $p$ is fixed and $n_1,n_2\to\infty$, the above generalized likelihood ratio test is asymptotically optimal.

High-dimensional datasets are ubiquitous in the contemporary era of Big Data.
As dimensions of modern data $p$ in genetics, signal processing, econometrics and other fields are often comparable to sample sizes $n$, the most significant challenge in high-dimensional data is that the fixed-$p$-large-$n$ setup prevalent in classical statistical inference is no longer valid.
Yet the philosophy remains true that statistical inference is only possible when the sample size relative to the \textit{true} parameter size is sufficiently large.
%Despite the growing dimension, in most applications, we only care
Most advances in high-dimensional statistical inference so far have been made under some `sparsity' conditions, i.e., all but a small (often vanishing) fraction of the $p$-dimensional model parameters are zero.
The assumption in effect reduces the parameter size to an estimable level, and it makes sense in many applications because often only few covariates are \textit{really} responsible for the response, though identification of these few covariates is still a nontrivial task.
In the high-dimensional regression setting $Y = X\beta + \epsilon$ where $Y \in \mathbb{R}^n$, $X \in \mathbb{R}^{n\times p}$, $\beta \in \mathbb{R}^p$ with $p, n \rightarrow \infty$ simultaneously, a common assumption to make is $k \log p /n \rightarrow 0$ with $k = \|\beta\|_0:=\sum_{j=1}^p \mathbbm{1}_{\{\beta_j\neq 0\}}$.  Therefore, $k$ is the true parameter size, which vanishes relative to the sample size $n$, and $\log p $ is understood as the penalty to pay for not knowing where the $k$ true parameters are.
%it is otherwise next to impossible to estimate $\beta$ without assuming some sparsity condition on $\beta$.

Aiming to take a step in studying the fundamental aspect of two-sample hypothesis testing in  high dimensions, this paper is primarily concerned with the following testing problem: we need to decide whether the responses in the two samples have different linear dependencies on the covariates.
More specifically, under the same regression setting as in \eqref{Eq:Model} with $\min\{p,n\} \rightarrow \infty$, we wish to test the global null hypothesis
\begin{equation}
\label{Eq:Null}
H_0:  \beta_1 = \beta_2
\end{equation}
against the composite alternative
\begin{equation}
\label{Eq:Alternative}
H_1:  \| \beta_1 - \beta_2\|_2 \geq 2\rho, \; \| \beta_1 - \beta_2 \|_0\leq k.
\end{equation}
In other words, we assume that under the alternative hypothesis, the difference between the two regression coefficients is a $k$-sparse vector with $\ell_2$ norm at least $2\rho$ (the additional factor of $2$ here exists to simplify relevant statements under the reparametrisation we will introduce later in Section~\ref{sec:method}).
Throughout this paper, we do not assume the sparsity of $\beta_1$ or $\beta_2$ under the alternative.

Classical $F$-tests no longer work well on the above testing problem, for the simple reason that it is not possible to get good estimates of $\beta$'s through naive least square estimators, which are necessary in establishing $\RSS$ in \eqref{Eqn:rss-def} to measure the model's goodness of fit.
%This is no longer possible in high-dimensional setting because $X_1^\top X_1$ is not invertible if $p > n_1$.
%Even in the case of $p < n_1$ with say, $p = n_1/2$, $\hat\beta_1 = (X_1^\top X_1)^{-1} X_1^\top Y_1$  is no longer consistent, resulting in unexpected behaviors in statistics $d_{\mathrm{chow}}$.
%\fn{I am not entirely sure about this, but the ratio of parameter number and sample size is constant, therefore the naive estimator cannot be consistent.}
A standard way out is to impose certain kinds of sparsity on both $\beta_1$ and $\beta_2$ to ensure that both quantities are estimable.
To our best knowledge, this is the out-of-shelf approach taken by most literature, see, for instance, \cite{stadler2012two,xia2015testing}.
Nevertheless, it is both more interesting and relevant in applications to study the testing problem where neither $\beta_1$ nor $\beta_2$ is estimable but only $\beta_1 - \beta_2$ is sparse.

Practically, the assumption that $\beta_1$ and $\beta_2$ are both dense, but their difference is sparse can be motivated by comparisons of paired high-dimensional datasets where the commonly seen sparsity assumption fails for each individual dataset. For instance, \citet{kraft2009genetic} pointed out that in some genetic studies, ``many, rather than few, variant risk alleles are responsible for the majority of the inherited risk of each common disease''. Hence, to compare the difference between two such populations, it may not be appropriate to assume that the number of responsible single nucleotide polymorphisms (SNPs) in each population is small. On the other hand, the difference between the two populations can still be accounted for by a few SNPs, as pointed out in the Framingham Offspring Study \citep{kannel1979diabetes, xia2018two}. The area of differential networks provides further examples to motivate two-sample testing of regression coefficients assuming only sparsity in their difference. Here, researchers are interested in whether two networks formulated as Gaussian graphical models, such as `brain connectivity network' and gene-gene interaction network \citep{xia2015testing,Chabonnier2015}, are different in two subgroups of population. Such complex networks are mostly of high-dimensional nature, in the sense that the number of nodes or features in the networks are large, relative to the number of observations.  Since the off-diagonal entries of the inverse covariance matrix in a graphical model can be equated to the node-wise regression coefficients, such differential network testing problems can be reduced to multiple two-sample high-dimensional regression coefficient testing problems. Such networks are often dense as interactions within different brain parts or genes are omnipresent, but because they are subject to the about same physiology, the differences between networks from two subpopulations are conceivably small, i.e., there are only a few different edges from one network to another.
In the above case of dense coefficients, sparsity assumption  may not be true, and it is impossible to obtain reasonable estimates of either regression coefficient $\beta_1$ or $\beta_2$ when $p$ is of the same magnitude as $n$.
For this reason, any approach to detect the difference between $\beta_1$ and $\beta_2$, which is built upon comparing estimates of $\beta_1$ and $\beta_2$ in some ways, fails.
In fact, any inference on $\beta_1$ or $\beta_2$ is not possible unless we make some other stringent structural assumptions on the model.
However, certain inference on the coefficient difference $\beta_1 - \beta_2$, such as testing the zero null with the sparse alternative, is feasible by exploiting sparse difference between different networks without many assumptions. 
See Section~\ref{sec:data} for an application of our method to a real-world single cell RNA-sequencing dataset, which exemplifies the aforementioned two-sample differential network analysis.
% , which we perform the aforementioned differential network analysis.
% model by Gaussian graphical models and perform two-sample testing in each nodewise regression to compare two types of immune T cells. 

\subsection{Related Works}
The two-sample testing problem in its most general form is not well-understood in high dimensions. Most of the existing literature has focused on testing the equality of means, namely the high-dimensional equivalence of \eqref{Eqn:classical-two-sample}, see, e.g. \citet{cai2014two,chen2019two}.
Similar to our setup, in the mean testing problems, we may allow for non-sparse means in each sample and test only for sparse differences between the two population means \citep{cai2014two}. The intuitive approach for testing equality of means is to eliminate the dense nuisance parameter by taking the difference in means of the two samples and thus reducing it to a one-sample problem of testing a sparse mean against a global zero null, which is also known as the `needle in the haystack' problem well studied previously by e.g.\ \citet{Ingster1997, donoho2004higher}.
%Then the high-dimensional one-sample testing methods kick in, which has been a well-studied topic (\cite{AriasCastroetal2011}).
Such reduction, however, is more intricate in the regression problem, as a result of different design matrices for the two samples.

%\fn{To be rewritten.}
Literature is scarce for two-sample testing under high-dimensional regression setting. \citet{stadler2012two}, \citet{xia2018two}, \citet{xia2020gap} have proposed methods that work under the additional assumption so that both $\beta_1$ and $\beta_2$ can be consistently estimated. \citet{Chabonnier2015} and \citet{ZhuBradic2016} are the only existing works in the literature we are aware of that allow for non-sparse regression coefficients $\beta_1$ and $\beta_2$.
%Both works are based on the following two-stage procedure: first apply dimension reduction (such as variable selection) to obtain a low-dimensional problem and then apply the classical $F$-tests to the said low-dimensional problem. To name a few examples,
Specifically, \cite{Chabonnier2015} look at a sequence of possible supports of $\beta_1$ and $\beta_2$ on a Lasso-type solution path and then apply a variant of the classical $F$-tests to the lower-dimensional problems restricted on these supports, with the test $p$-values adjusted by a Bonferroni correction.
%\cite{stadler2012two} performs the variable selection stage by applying a Lasso-type estimator on a subset of the sample and evaluate the variable selection by a significance test on the sample left out of the screening.
\cite{ZhuBradic2016} (after some elementary transformation) uses a Dantzig-type selector to obtain an estimate for $(\beta_1+\beta_2)/2$ and then use it to construct a test statistic based on a specific moment condition satisfied under the null hypothesis. As both tests depend on the estimation of nuisance parameters, their power can be compromised if such nuisance parameters are dense.

\subsection{Our contributions}
\label{Sec:OurContribution}
%We observe a phase transition when the difference vector $\theta$ has different levels of sparsity.
Our contributions are four-fold. First, we propose a novel method to solve the testing problems formulated in~\eqref{Eq:Null} and~\eqref{Eq:Alternative} for model~\eqref{Eq:Model}.
Through `complementary sketching', which is a delicate linear transformation on both the designs and responses, our method turns the testing problem with two different designs into one with the same design of dimension $m\times p$ where $m = n_1 + n_2 - p$.
After taking the difference in two regression coefficients, the problem is reduced to testing whether the coefficient in the reduced one-sample regression is zero against sparse alternatives. % with coefficient $\theta = (\beta_1 - \beta_2)/2$.
The transformation is carefully chosen such that the error distribution in the reduced one-sample regression is homoscedastic. %namely, $N(0, \sigma^2I)$.
This paves the way for constructing test statistics using the transformed covariates and responses. Our method is easy to implement and does not involve any complications arising from solving computationally expensive optimization problems.
Moreover, when complementary sketching is combined with any methods designed for one-sample global testing problems \citep[e.g.][]{ITV10, AriasCastroetal2011, carpentier2019minimax, carpentier2021optimal}, our proposal substantially supplies a novel class of testing and estimation procedures for the corresponding two-sample problems. However, as the design matrices after the complementary sketching transformation possess complex dependence structure, theoretical results from the one-sample testing literature cannot be directly applied, and new techniques are required in the current work to analyse our two-stage procedure.

Our second contribution is that, in the sparse regime, where the sparsity parameter $k \sim p^\alpha$ in the alternative \eqref{Eq:Alternative} for any fixed $\alpha \in (0, 1/2)$, we show that the detection limit of our procedure, defined as the minimal $\|\beta_1-\beta_2\|_2$ necessary for asymptotic almost sure separation of the alternative from the null, is minimax optimal up to a multiplicative constant under a Gaussian design. More precisely, we show that in the asymptotic regime where $n_1,n_2,p$ diverge at a fixed ratio, and for a large class of covariance matrices of the design, if $\rho^2 \gtrsim \frac{k\log p}{ n\kappa_1}$, where $\kappa_1$ is a constant depending on $n_1/n_2$ and $p/m$ only, then our test has asymptotic power \(1\) almost surely. On the other hand, in the same asymptotic regime, if $\rho^2 \leq \frac{c_{\alpha}  k\log p}{n\kappa_1}$ for some $c_{\alpha}$ depending only on $\alpha$, then almost surely no test has asymptotic size \(0\) and power \(1\).

Furthermore, our results reveal the effective sample size of the two-sample testing problem. Here, by effective sample size, we mean the sample size for a corresponding one-sample testing problem (i.e.\ testing $\beta=0$ in a linear model $Y= X\beta+\epsilon$ with rows of $X$ following the same distribution as rows of $X_1$ and $X_2$) that has an asymptotically equal detection limit; see the discussion after Theorem~\ref{Thm:LowerBound} for a detailed definition. At first glance, one might think that the effective sample size is $m$, which is the number of rows in the reduced design. This hints that the reduction to the one-sample problem has made the original two-sample problem obsolete. However, on deeper thoughts, as an imbalance in the numbers of observations in $X_1$ and $X_2$ clearly makes testing more difficult, the effective sample size has to also incorporate this effect.  We see from the previous point that uniformly for any $\alpha$ less than and bounded away from $1/2$, the detection boundary is of order $\rho^2 \asymp \frac{k \log p}{n \kappa_1}$, with the precise definition of $\kappa_1$ given in Proposition~\ref{Prop:WtW-general}. Writing $n_1/n_2 = r$ and $p/m = s$, our results on the sparse case implies that the two-sample testing problem has the same order of detection limit as in a one-sample problem with sample size $n\kappa_1 = m(r^{-1} + r +2)^{-1}$. We note that this effective sample size is proportional to $m$, and for each fixed $m$, maximized when $r=1$ (i.e. $n_1=n_2$) and approaches $m/n$ in the most imbalanced design. This is in agreement with the intuition that testing is easiest when $n_1 = n_2$ and impossible when $n_1$ and $n_2$ are too imbalanced. Our study, thus, sheds light on the intrinsic difference between two-sample and one-sample testing problems and characterizes the precise dependence of the difficulty of the two-sample problem on the sample size and dimensionality parameters.

Finally, we observe a phase transition phenomena of how the minimax detection limit depends on the sparsity parameter $k$.  On top of minimax rate optimal detection limit of our procedure in the sparse case when $k \asymp p^\alpha$ for $\alpha\in [0,1/2)$, we also prove that a modified version of our procedure, designed for denser signals, is able to achieve minimax optimal detection limit up to logarithmic factors in the dense regime $k\asymp p^{\alpha}$ for $\alpha\in (1/2,1)$. However, the detection limit is of order $\rho^2 \asymp \frac{k\log p}{n\kappa_1}$ in the sparse regime, but of order $\rho^2 \asymp p^{-1/2}$ up to logarithmic factors in the dense regime. Such a phase transition phenomenon is qualitatively similar to results previously reported in the one-sample testing problem \citep[see, e.g.][]{ITV10, AriasCastroetal2011, carpentier2019minimax, carpentier2021optimal}.

\subsection{Organization of the paper}
We describe our methodology in detail in Section~\ref{sec:method} and establish its theoretical properties in Section~\ref{sec:theory}. Numerical results illustrate the finite sample performance of our proposed algorithm in Section~\ref{sec:simu}.
We present in Section~\ref{sec:data} a real data example to compare gene regulatory networks in two close-related types of T cells.
Proofs of our main results are deferred until Section~\ref{sec:proofs} with ancillary results in Section~\ref{sec:ancillary}.

\subsection{Notation}
For any positive integer $n$, we write $[n] := \{1, \dots, n\}$. For a vector $v = (v_1, \ldots, v_n)^\top \in \mathbb{R}^n$, we define $\|v\|_0 := \sum_{i=1}^n \mathbbm{1}_{\{v_i \neq 0\}}$, $\|v\|_\infty := \max_{i \in [n]} |v_i|$ and $\|v\|_q:=\bigl\{\sum_{i=1}^{n} (v_i)^q\bigr\}^{1/q}$ for any positive integer $q$, and let $\mathcal{S}^{n-1} := \{ v \in \mathbb{R}^n : \| v\|_2 = 1\}$. The support of vector $v$ is defined by $\supp(v) : = \{i\in [n]: v_i \neq 0 \}$.

For $n\geq m$, $\mathbb{O}^{n\times m}$ denotes the space of $n\times m$ matrices with orthonormal columns. For $a \in \mathbb{R}^p$, we define $\diag(a)$ to the $p\times p$ diagonal matrix with diagonal entries filled with elements of $a$, i.e., $(\diag(a))_{i,j} = \mathbbm{1}_{\{ i = j\}} a_i$. Let $A \in \mathbb{R}^{p \times p}$, and we write $\|A\|_{\mathrm{op}}$, $\|A\|_{\mathrm{F}}$ and $\|A\|_{\max}$ for its operator, Frobenius and entrywise $\ell_\infty$ norm respectively. We define $\diag(A)$ to be the $p \times p$ diagonal matrix with diagonal entries coming from $A$, i.e., $(\diag(A))_{i,j} = \mathbbm{1}_{\{i = j\}} A_{i,j}$. We also write $\mathrm{tr}(A) := \sum_{i\in[p]} A_{i,i}$.
For a symmetric matrix $A \in\mathbb{R}^{p\times p}$ and $j\in[p]$, we write $\lambda_j(A)$ for its $j$th largest (real) eigenvalue. When $\lambda_p(A)\geq 0$, $A$ is positive semidefinite, which we denote by $A\succeq 0$. For $A$ symmetric and $k\in[p]$, the $k$-sparse operator norm of $A$ is defined by
\[
\|A\|_{k,\mathrm{op}}:= \sup_{v\in\mathcal{S}^{p-1}:\|v\|_0\leq k} |v^\top A v|.
\]
For any $S\subseteq [n]$, we write $v_S$ for the $|S|$-dimensional vector obtained by extracting coordinates of $v$ in $S$ and $A_{S,S}$ the matrix obtained by extracting rows and columns of $A$ indexed by $S$.

Given two sequences $(a_n)_{n\in\mathbb{N}}$ and $(b_n)_{n\in\mathbb{N}}$ such that $b_n > 0$ for all $n$, we write $a_n = \mathcal{O}(b_n)$ if $|a_n| \leq C b_n$ for some constant $C$. If the constant $C$ depends on some parameter $x$, we write $a_n = \mathcal{O}_x (b_n)$ instead. Also, $a_n = \smallO(b_n)$ denotes $a_n/b_n \to 0$.

%!TEX root = ../compsket.tex

\section{Testing via complementary sketching}
\label{sec:method}
In this section, we describe our testing strategy. Since we are only interested in the difference in regression coefficients in the two linear models, we reparametrize~\eqref{Eq:Model} with $\gamma := (\beta_1 + \beta_2)/2$ and $\theta := (\beta_1 - \beta_2)/2$ to separate the nuisance parameter from the parameter of interest. Define
\[
\Theta_{p,k}(\rho):=\bigl\{\theta\in\mathbb{R}^p: \|\theta\|_2\geq \rho \; \text{and}\; \|\theta\|_0\leq k\bigr\}.
\]
Under this new parametrization, the null and the alternative hypotheses can be equivalently formulated as
\[
H_0: \theta = 0 \quad \text{and}\quad H_1: \theta\in\Theta_{p,k}(\rho).
\]
The parameter of interest $\theta$ is now $k$-sparse under the alternative hypotheses. However, its inference is confounded by the possibly dense nuisance parameter $\gamma\in\mathbb{R}^p$. A natural idea, then, is to eliminate the nuisance parameter from the model. In the special design setting where $X_1 =  X_2$ (in particular, $n_1 =n_2$), this can be achieved by considering the sparse regression model $Y_1 - Y_2 =  X_1\theta + (\epsilon_1 - \epsilon_2)$. While the above example only works in a special, idealized setting, it nevertheless motivates our general testing procedure.

To introduce our test, we first concatenate the design matrices and response vectors to form
\[
 X = \begin{pmatrix}  X_1\\ X_2\end{pmatrix}
\quad \text{and} \quad
 Y = \begin{pmatrix}  Y_1\\ Y_2\end{pmatrix}.
\]
A key idea of our method is to project $X$ and $Y$ respectively along $n-p$ pairs of directions that are complementary in sign in a subset of their coordinates, a process we call \textit{complementary sketching}. Specifically, assume $n_1+n_2>p$ and define $n:=n_1+n_2$ and $m := n - p$ and let $A_1\in\mathbb{R}^{n_1\times m}$ and $A_2\in\mathbb{R}^{n_2\times m}$ be chosen such that
\begin{equation}
\label{Eq:ConditionA}
A_1^\top A_1 + A_2^\top A_2 = I_m \quad \text{and} \quad A_1^\top X_1 + A_2^\top X_2 = 0.
\end{equation}
In other words, $A := (A_1^\top, A_2^\top)^\top$ is a matrix with orthonormal columns orthogonal to the column space of $X$. Such $A_1$ and $A_2$ exist since the null space of $X$ has dimension at least $m$. Define $Z:= A_1^\top Y_1 + A_2^\top Y_2 \in \mathbb{R}^m$, $W:=A_1^\top X_1- A_2^\top  X_2 \in \mathbb{R}^{m\times p}$ and \( \xi = A_1^\top \varepsilon_1 + A_2^\top \varepsilon_2 \in \mathbb{R}^m\).  From the above construction, we have
\begin{equation}
\label{Eq:ZW}
Z = A_1^\top X_1\beta_1 +  A_2^\top  X_2 \beta_2 + (A_1^\top \epsilon_1 + A_2^\top \epsilon_2) = W \theta + \xi,
\end{equation}
where $\xi\mid W \sim N_m(0, A^\top A) = N_m(0, \sigma^2 I_m)$. We note that similar to conventional sketching \citep[see, e.g.][]{mahoney2011randomized}, the complementary sketching operation above synthesizes $m$ data points from the original $n$ observations. However, unlike conventional sketching, where one projects the design $X$ and response $Y$ by the same sketching matrix $S\in\mathbb{R}^{m\times n}$ to obtain sketched data $(SX, SY)$, here we project $X$ and $Y$ along different directions to obtain $(\tilde A^\top X, A^\top Y)$, where $\tilde A:= (A_1^\top, -A_2^\top)^\top$ is complementary in sign to $A$ in its second block. Moreover, the main purpose of the conventional sketching is to trade off statistical efficiency for computational speed by summarizing raw data with a smaller number of synthesized data points, whereas the main aim of our complementary sketching operation is to eliminate the nuisance parameter, and surprisingly, as we will see in Section~\ref{sec:theory}, there is essentially no loss of statistical efficiency introduced by our complementary sketching in this two-sample testing setting. 

To summarize, after projecting $X$ and $Y$ via complementary sketching to obtain $W$ and $Z$, we reduce the original two-sample testing problem to a one-sample problem with $m$ observations, where we test the global null of $\theta = 0$ against sparse alternatives using data $(W, Z)$. From here, we can construct test statistics as functions of $W$ and $Z$, for which we describe two different tests. The first testing procedure, detailed in Algorithm~\ref{Algo:Test}, computes the sum of squares of hard-thresholded inner products between the response $Z$ and standardized columns of the design matrix $W$ in~\eqref{Eq:ZW}. We denote the output of Algorithm~\ref{Algo:Test} with input $X_1$, $X_2$, $Y_1$ and $Y_2$ and tuning parameters $\omega$ and $\tau$ as $\psi^{\mathrm{sparse}}_{\omega,\tau}(X_1,X_2,Y_1,Y_2)$. As we will see in Section~\ref{sec:theory}, if we have a `good' estimator $\hat\sigma$ for the noise level $\sigma$, the choice of  $\omega=2\hat\sigma\sqrt{\log p}$ and $\tau = k\hat\sigma^2\log p$ would be suitable for testing against sparse alternatives in the case of $k\leq p^{1/2}$. On the other hand, in the dense case when $k > p^{1/2}$, one option would be to choose $\omega=0$. However, it turns out to be difficult to set the test threshold level $\tau$ in this dense case using the known problem parameters. Therefore, we decided to study instead the following as our second test. We apply steps 1 to 4 of Algorithm~\ref{Algo:Test} to obtain the vector $Z$, and then for a suitable choice of threshold level $\eta$, define our test as
\[
    \psi^{\mathrm{dense}}_{\eta}(X_1,X_2,Y_1,Y_2) := \mathbbm{1}{\{\|Z\|_2^2 \geq \eta\}}.
\]

\begin{algorithm}[htbp]
	\KwIn{$X_1 \in \mathbb{R}^{n_1\times p}, X_2 \in \mathbb{R}^{n_2\times p}, Y_1 \in \mathbb{R}^{n_1}, Y_2\in \mathbb{R}^{n_2}$ satisfying $n_1+n_2-p > 0$, a hard threshold level $\omega \geq 0$, and a test threshold level $\tau > 0$.}
    Set $m \leftarrow n_1+n_2-p$.\\
    Form $A\in\mathbb{O}^{n\times m}$ with columns orthogonal to the column space of $(X_1^\top, X_2^\top)^\top$.\\
    Let $A_1$ and $A_2$ be submatrices formed by the first $n_1$ and last $n_2$ rows of $A$.\\
    Set $Z\leftarrow A_1^\top Y_1 + A_2^\top Y_2$ and $W \leftarrow A_1^\top X_1- A_2^\top  X_2$.\\
    %Compute $Q_j \leftarrow W_j^\top Z / \|W_j\|_2$, where $W_j$ is the $j$th column of $W$.\\
    Compute $Q \leftarrow \{\diag(W^\top W)\}^{-1/2}W^\top Z $.\\
    Compute the test statistic
    \[
    T:= \sum_{j=1}^p Q_j^2\mathbbm{1}_{\{|Q_j| \geq \omega\}}.
    \]\\
%    Set $\tilde\sigma \leftarrow \|Z\|_2/\sqrt{n}$ and compute $H^*:=\max\{H(t): t\in[1, (5\log p)^{1/2}]\cap \mathbb{N}\}$, where
%    \[
%    H(t) := \frac{|\{j: |Z^\top W_j| > \tilde\sigma t\|W_j\|_2\}| - 2p\bar\Phi(t)}{\{2p\bar\Phi(t)(1-2\bar\Phi(t))\}^{1/2}}.
%    \]\\
    Reject the null hypothesis if $T \geq \tau$.
	\caption{Pseudo-code for complementary sketching-based test $\psi^{\mathrm{sparse}}_{\omega,\tau}$.}
	\label{Algo:Test}
\end{algorithm}
%\section{Design conditions}
%We clarify the design conditions where our Algorithm~\ref{Algo:Test} works.
The computational complexity of both $\psi^{\mathrm{sparse}}_{\omega,\tau}$ and $\psi^{\mathrm{dense}}_{\eta}$ depends on Step~2 of Algorithm~\ref{Algo:Test}. In practice, we can form the projection matrix $A$ as follows. We first generate an $n\times m$ matrix $M$ with independent $N(0,1)$ entries, and then project columns of $M$ to the orthogonal complement of the column space of $X$ to obtain $\tilde M := (I_n - XX^\dagger)M$, where $X^\dagger$ is the Moore--Penrose pseudoinverse of $X$. Finally, we extract an orthonormal basis from the columns of $\tilde M$ via a QR decomposition $\tilde M = AR$, where $R$ is upper triangular and $A$ is a (random) $n\times m$ matrix with orthonormal columns that can be used in Step~2 of Algorithm~\ref{Algo:Test}. The overall computational complexity for our tests are therefore of order $\mathcal{O}(n^2p + nm^2)$. Finally, it is worth emphasizing that while the matrix $A$ generated this way is random, our test statistics $T=\sum_{j=1}^pQ_j^2\mathbbm{1}_{\{|Q_j|\geq \omega\}}$ and $\|Z\|_2^2$, are in fact deterministic. To see this, we observe that both
\begin{align*}
W^\top Z &= (A_1^\top X_1 - A_2^\top X_2)^\top (A_1^\top Y_1 + A_2^\top Y_2) \\
&= \begin{pmatrix}X_1^\top & -X_2^\top\end{pmatrix}\begin{pmatrix}A_1\\A_2\end{pmatrix}\begin{pmatrix}A_1^\top & A_2^\top\end{pmatrix}\begin{pmatrix}Y_1\\Y_2\end{pmatrix}
\end{align*}
and $\|Z\|_2^2 = Y^\top AA^\top Y$ depend on $A$ only through $AA^\top$, which is determined by the column space of $A$. Moreover, by Lemma~\ref{Lemma:GramW}, $(\|W_j\|^2_2)_{j\in[p]}$, being diagonal entries of $W^\top W = 4X_1^\top A_1A_1^\top X_1$, are also functions of $X$ alone. This attests that both test statistics, and consequently our two tests, are deterministic in nature.

%!TEX root = ../compsket.tex

\section{Theoretical analysis}
\label{sec:theory}
We now turn to the analysis of the theoretical performance of $\psi^{\mathrm{sparse}}_{\omega,\tau}$ and $\psi^{\mathrm{dense}}_{\eta}$. We consider both the size and power of each test, as well as the minimax lower bounds for smallest detectable signal strengths.

In addition to working under the regression model~\eqref{Eq:Model}, we further assume the following conditions in our theoretical analysis. For some constants $0< \underline{\lambda} \leq \overline{\lambda} < \infty$, we write
\begin{align*}
&\mathcal{C}:=\{\Sigma \in \mathbb{R}^{p\times p}: \text{$\Sigma \succeq 0$,  $\mathrm{diag}(\Sigma)=I_p$ and for all  $S\subseteq[p]$ with $|S|=k$}\\
& \hspace{6cm} \underline{\lambda}\leq \lambda_k(\Sigma_{S,S})\leq \lambda_1(\Sigma_{S,S})\leq \overline{\lambda}\}
\end{align*}
\begin{enumerate}[label={(C\arabic*)}, series=l_cond]
\item All rows of $X_1$ and $X_2$ are independent and follows $N_p(0, \Sigma)$ distribution such that $\Sigma\in \mathcal{C}$. \label{cond:design}
\item Parameters $n_1,n_2,p$ satisfy $m=n_1+n_2-p > 0$ and lie in the asymptotic regime where $n_1/n_2\to r$ and $p/m\to s$ as $n_1,n_2,p\to\infty$. \label{cond:asymp}
\end{enumerate}
The condition that $\diag(\Sigma)=I_p$ in \ref{cond:design} means that all columns of the design matrix should have unit variance and that $\Sigma$ is in fact a correlation matrix. This is assumed here both to simplify notation in our theoretical analysis and to reflect the common practice of column normalization in practical applications (especially when covariates are measured in different units). For a generic $\Sigma$, we remark that the testing boundary should be measured in terms of $\theta^\top \diag(\Sigma)\theta$ instead of $\|\theta\|_2^2$ and results similar to Theorems~\ref{Thm:Null}, \ref{Thm:Alternative}, \ref{Thm:LowerBound}, \ref{Thm:AlternativeDense}, \ref{Thm:LowerBoundDense} can be derived via the reduction $X \mapsto X \{\diag(\Sigma)\}^{-1/2}$. The condition in~\ref{cond:design} that the spectrum of any $k\times k$ principal submatrix of $\Sigma$ is contained in $[\underline{\lambda},\overline{\lambda}]$ is relatively mild. It requires that any $k$ covariates are not too collinear. We note that it is in particular implied if $\Sigma$ itself has a bounded condition number, or alternatively if $\Sigma$ satisfies the restricted isometry condition.  %The main purpose for assuming \ref{cond:design} is to ensure that the sketched design matrix $W$ satisfies Proposition~\ref{Prop:WtW-general}. In fact, as revealed in the proof of our theory, even if matrices $X_1$ and $X_2$ do not follow the Gaussian design, as long as $W$ satisfies \eqref{Eq:WtWdiag-1} and~\eqref{Eq:kopNorm-1}, all results in this section will still be true.

The condition $n_1+n_2-p>0$ in \ref{cond:asymp} is necessary in this two-sample problem, since otherwise, for any prescribed value of $\Delta := \beta_1-\beta_2$,
the equation system with $\beta_1$ as unknowns
\begin{equation*}
    \begin{pmatrix}
        X_1 \\
    X_2
    \end{pmatrix}
    \beta_1 = \begin{pmatrix}
        Y_1\\
        Y_2 -X_2 \Delta
    \end{pmatrix}
\end{equation*}
has at least one solution when $(X_1^\top, X_2^\top)^\top$ has rank $n$. As a result, except in some pathological cases, we can always find $\beta_1,\beta_2\in\mathbb{R}^p$ that fit the data perfectly with $Y_1=X_1\beta_1$ and $Y_2=X_2\beta_2$, which makes the testing problem impossible. A more rigorous statement regarding the necessity of this condition in a minimax sense is proved in Proposition~\ref{Prop:MXeq1}.
%Some existing literature permits $n_1 +n_2 < p$, but their methods are only applicable under structural assumptions that in essence reduce the true degree of freedom in parameters to an estimable level.
Finally, we have carried out proofs of our theoretical results with finite sample arguments wherever possible. Nevertheless, due to a lack of finite-sample bounds on the empirical spectral density of matrix-variate Beta distributions, all results in this section are presented under the asymptotic regime set out in Condition~\ref{cond:asymp}. Under this condition, we were able to exploit existing results in the random matrix theory to obtain a sharp dependence of the detection limit on $s$ and~$r$.

In practice, the noise variance $\sigma^2$ is typically unknown. The problem of noise variance estimation in high-dimensional linear models is a well-studied one itself that has received much attention recently \citep{fan2012variance, sun2012scaled, homrighausen2013lasso, Dicker2014, reid2016study}. As the main focus of the current work is on two-sample testing of regression coefficients, we will make the simplifying assumption in our theoretical analysis that the noise variance $\sigma^2$ is either known or that good estimators exist. Specifically, we assume that one of the following two conditions about the noise variance is met:
\begin{enumerate}[label={(S\arabic*)}] %, resume*=l_cond]
\item There exists an independent estimator $\hat{\sigma}$ of $\sigma$ such that $\hat{\sigma} \xrightarrow{\text{a.s.}} \sigma$.  
    \label{cond:sigma}
    \item There exists an independent estimator $\hat{\sigma}$ such that we have $|\hat\sigma/\sigma - 1 | = \smallO(p^{-1/2}\log^{1/2}p)$ almost surely.
        \label{cond:sigma-strong}
\end{enumerate}
For most of our theoretical analysis, the much weaker condition~\ref{cond:sigma} suffices. The stronger condition~\ref{cond:sigma-strong} is needed only in Theorem~\ref{Thm:AlternativeDense}, where we derive the upper bound for the test $\psi^{\mathrm{dense}}_{\eta}$. 

We remark that condition~\ref{cond:sigma} is very mild.  It is for instance significantly weaker than most conditions where one requires the rate of convergence of $\hat{\sigma}$ of order at least a polynomial in $p$, i.e., for some $\alpha \le 1/2$ and any $\varepsilon>0$,
\(
    \mathbb{P} ( | \hat{\sigma}/\sigma -1| p^\alpha >  \varepsilon) = \mathcal{O}(1).
\)
% We take $p^{-2}$ as the probability bound in \ref{cond:sigma} to make certain that the ubiquitous almost-sure analysis in our theory works when $\sigma$ is unknown.  
The independence assumption on $\hat{\sigma}$ is often achieved in practice by sample splitting. For instance, using consistent estimators of $\sigma$ proposed in \citet{Dicker2014}, we may use $\mathcal{O}(\log^2 n)$ data points from each sample to estimate $\sigma$, and use the remaining samples to perform the hypothesis test. However, we will assume that $\hat\sigma$ is available independent of $(X_1,Y_1,X_2,Y_2)$ so that we can focus our theoretical analysis on the primary points of interest in this problem.

Condition \ref{cond:sigma-strong} is much stronger than \ref{cond:sigma} but slightly weaker than the usual $\sqrt{n}$-consistency found in the parametric literature, albeit we need the convergence to take place almost surely.
Similar to \ref{cond:sigma}, in reality we could possibly obtain such an estimator via the usual sample-splitting argument, where we take a fixed proportion of all data points to estimate $\sigma$ and the rest for testing.

Finally, for ease of reference, we summarize all our theoretical findings in Table~\ref{tab:all-bounds}. Here, the lower bounds are proved in a subclass of covariance matrices $\mathcal{C}(D)\subseteq \mathcal{C}$ defined in~\eqref{Eq:CD}.
\begin{comment}
\begin{table}[htp]
  \begin{center}
\begin{tabular}{rcc}
  \hline\hline
             & sparse &  dense \\ 
             \hline 
upper bound &       \(\displaystyle \frac{7\sigma^2 k \log p}{\underline{\lambda}^2 n \kappa_1}\)                 &                  \(\displaystyle\frac{2\sigma^2\sqrt{m \log  p}}{\underline{\lambda} n \kappa_1}\)\\[.6cm]
lower bound  &          \(\displaystyle\frac{ (1- 2\alpha - \epsilon)\sigma^2 k \log p}{8 D n \kappa_1}\)               & \(\smallO(p^{-1/2} \min\{ \log^{-1/2}(ep/k), D^{-3/2}\})\)\\
\hline\hline
\end{tabular}
\caption{\label{tab:all-bounds-r}Comparison of upper and lower bounds of the testing boundary in terms of $\rho^2$ in both sparse and dense cases. }
\end{center}
\end{table}
\end{comment}

\begin{table}[htp]
  \begin{center}
\begin{tabular}{rcc}
  \toprule
             & sparse &  dense  \\
             \midrule
upper bound &       \(\displaystyle \frac{7\sigma^2 k \log p}{\underline{\lambda}^2 n \kappa_1}\)                 &                  \(\displaystyle\frac{2\sigma^2\sqrt{m \log  p}}{\underline{\lambda} n \kappa_1}\)\\[.4cm]
lower bound  &          \(\displaystyle\frac{ (1- 2\alpha - \epsilon)\sigma^2 k \log p}{8 D n \kappa_1}\)               & \(\smallO(p^{-1/2} \min\{ \log^{-1/2}(ep/k), D^{-3/2}\})\)\\
\bottomrule
\end{tabular}
\caption{\label{tab:all-bounds}Comparison of upper and lower bounds of the testing boundary in terms of $\rho^2$ in both sparse and dense cases. }
\end{center}
\end{table}

\subsection{Sparse case}
We consider in this subsection the test $\psi^{\mathrm{sparse}}_{\omega,\tau}$, which is suitable for distinguishing $\beta_1$ and $\beta_2$ that differ in only a few coordinates, the setting that has more subtle phenomena and hence is most interesting to us. Our first result below states that with a choice of hard-thresholding level $\omega$ of order $\sigma\sqrt{\log p}$, the test has asymptotic size 0.

\begin{thm}
\label{Thm:Null}
If Conditions~\ref{cond:asymp} and~\ref{cond:sigma}  hold and $\beta_1=\beta_2$, then, with the choice of parameters $\tau > 0$ and $\omega = \hat{\sigma}\sqrt{(4+\varepsilon)\log p}$ for any $\varepsilon>0$, we have
\[
%\mathbb{P}\bigl\{
    \psi^{\mathrm{sparse}}_{\omega,\tau}(X_1, X_2, Y_1, Y_2) \xrightarrow{\mathrm{a.s.}} 0.
%\bigm| X_1, X_2\bigr\} \to 1.
\]
\end{thm}
The almost sure statement in Theorem~\ref{Thm:Null} and subsequent results in this section are with respect to both the randomness in $X = (X_1^\top, X_2^\top)^\top$ and in $\epsilon = (\epsilon_1^\top, \epsilon_2^\top)^\top$. However, a closer inspection of the proof of Theorem~\ref{Thm:Null} tells us that the statement is still true if we allow an arbitrary sequence of matrices $X$ (indexed by $p$) and only consider almost sure convergence with respect to the distribution of $\epsilon$.

The control of the asymptotic power of $\psi^{\mathrm{sparse}}_{\omega,\tau}$ is more involved. A key step in the argument is to show that $W^\top W$ is suitably close to a multiple of the population covariance matrix $\Sigma$. More precisely, in Proposition~\ref{Prop:WtW-general} below, we derive entrywise and $k$-sparse operator norm controls of the Gram matrix of the design matrix sketch $W$.

\begin{prop}
    \label{Prop:WtW-general}
    Under Conditions~\ref{cond:design} and \ref{cond:asymp}, we further assume  $k\in[p]$ and let $W$ be defined as in Algorithm~\ref{Algo:Test}. Then with probability 1,
    \begin{equation}
    \label{Eq:WtWdiag-1}
    \max_{j\in[p]} \biggl|\frac{(W^\top W)_{j,j}}{4n\kappa_1} - 1\biggr| \to 0,
    \end{equation}
    where $\kappa_1 :=r/\{(1+r)^2(1+s)\}$.
    Moreover, define $\tilde{W} := W \{\diag(W^\top W)\}^{-1/2}$. %such that         $\tilde{W}_{j,\ell} := W_{j,\ell}/(\sum_{j' = 1}^p W_{j',\ell}^2)^{1/2}$.
    If
    \begin{equation}
    \label{Eq:klogpovern-1}
    \frac{k\log(ep/k)}{n} \to 0,
    \end{equation}
    then there exists $C_{s,r}>0$, depending only on $s$ and $r$, such that     with probability 1, the following holds for all but finitely many $p$:
    \begin{equation}
    \|\tilde W^\top \tilde W - \Sigma\|_{k,\mathrm{op}} \leq C_{s,r}\biggl\{\overline{\lambda}\sqrt\frac{k\log(ep/k)}{n}+\overline{\lambda}^2\sqrt\frac{\log p}{n}\biggr\}.
    \label{Eq:kopNorm-1}
    \end{equation}
\end{prop}

%\textbf{Comment on $W$ being a deterministic function of $X$, so even though the statements are with a random $X$, for a specific given design matrix $X$, as long as we can check conditions~\eqref{Eq:WtWdiag} and~\eqref{Eq:kopNorm} are satisfied, all theorems hold true for that specific deterministic $X$ (tricky since statements are asymptotic). Also, comment on the asymptotic nature of the statement as this comes from asymptotic spectral distribution of Beta random matrices. In addition, comment on the Gaussianity assumption used here.}
We note that condition~\eqref{Eq:klogpovern-1} is relatively mild and would be satisfied if $k \leq  p^\alpha$ for any $\alpha \in [0, 1)$. As we will see later, the quantity $n\kappa_1$ in~\eqref{Eq:WtWdiag-1} can be viewed as the `effective sample size' of the two-sample testing problem. A more detailed discussion about the intuition and interpretation of this quantity is provided after Theorem~\ref{Thm:LowerBound}.

%As a consequence of Lemma~\ref{Lemma:GramW}, $W^\top W$ and its variant $\tilde{W}^\top \tilde{W}$ are deterministic functions of $X$.
All theoretical results in this section except for Theorem~\ref{Thm:Null} assume the random design Condition~\ref{cond:design} to hold. However, as revealed by the proofs, for any given (deterministic) sequence of $X$, these results remain true as long as \eqref{Eq:WtWdiag-1} and \eqref{Eq:kopNorm-1} are satisfied. The asymptotic nature of Proposition~\ref{Prop:WtW-general} is a result of our application of \citet[Theorem~1.1]{BaiHuPanZhou2015}, which guarantees an almost sure convergence of the empirical spectral distribution of Beta random matrices in the weak topology. This sets the tone for the asymptotic nature of our results, which depend on the aforementioned limiting spectral distribution.

The following theorem provides power control of our procedure $\psi^{\mathrm{sparse}}_{\omega,\tau}$, when the $\ell_2$ norm of the scaled difference in regression coefficient $\theta=(\beta_1-\beta_2)/2$ exceeds an appropriate threshold.
\begin{thm}
\label{Thm:Alternative}
Under Conditions~\ref{cond:design}, \ref{cond:asymp} and~\ref{cond:sigma}, we further assume $k\in[p]$ and that~\eqref{Eq:klogpovern-1} holds. If $\theta=(\beta_1-\beta_2)/2 \in \Theta_{p,k}(\rho)$ with $\rho^2\geq \frac{7\sigma^2 k \log p}{\underline{\lambda}^2n\kappa_1}$, and we set input parameters $\omega=\hat{\sigma}\sqrt{(4+\varepsilon)\log p}$ for any $\varepsilon\in(0,1)$ and $\tau\leq \hat\sigma^2 k\log p$ in Algorithm~\ref{Algo:Test}, then
\[
\psi^{\mathrm{sparse}}_{\omega,\tau}(X_1,X_2,Y_1,Y_2)\xrightarrow{\mathrm{a.s.}} 1.
\]
\end{thm}
The size and power controls in Theorems~\ref{Thm:Null} and \ref{Thm:Alternative} jointly provide an upper bound on the minimax detection threshold. Specifically, let $P_{\beta_1, \beta_2}^X$ be the conditional distribution of $Y_1,Y_2$ given $X_1, X_2$ under model~\eqref{Eq:Model}.
Conditionally on the design matrices $X_1$ and $X_2$ and given $k\in[p]$ and $\rho > 0$, the (conditional) \emph{minimax risk} of testing $H_0: \beta_1=\beta_2$ against $H_1: \theta=(\beta_1-\beta_2)/2 \in\Theta_{p,k}(\rho)$ is defined as
\[
    \mathcal{M}_X(k,\rho) := \inf_{ \psi} \biggl\{\sup_{\beta\in\mathbb{R}^p} P_{\beta,\beta}^X(\psi\neq 0) + \hspace{-0.3cm} \sup_{\substack{\beta_1,\beta_2\in\mathbb{R}^p\\ (\beta_1-\beta_2)/2\in\Theta_{p,k}(\rho)}} \hspace{-0.3cm} P_{\beta_1,\beta_2}^X(\psi \neq 1)\biggr\},
\]
where we suppress all dependences on the dimension of data for notational simplicity and the infimum is taken over all $\psi: (X_1, Y_1, X_2, Y_2) \mapsto \{0,1\}$.
%\fn{[Shall we use $\mathbb{P}$ instead of $P$ in the above def of $\mathcal{M}$? Or is it intentional to distinguish between `absolute' and conditional probability?]}
If $\mathcal{M}_X(k,\rho) \xrightarrow{\mathrm{p}} 0$, there exists a test $\psi$ that with asymptotic probability $1$ correctly differentiates the null and the alternative.  On the other hand, if $\mathcal{M}_X(k,\rho) \xrightarrow{\mathrm{p}} 1$, then asymptotically no test can do better than a random guess. The following corollary provides an upper bound on the signal size $\rho$ for which the minimax risk is asymptotically zero.

%If the signal strength measured by the $l_2$-norm of the $\theta$ exceeds $(2k\log p/(n \kappa_1))^{1/2}$, $\psi^{\mathrm{sparse}}$ has asymptotic size $0$ and asymptotic power $1$ with proper choices of $\omega$ and $\tau$ and thus this is the detectable regime.
\begin{cor}
\label{Cor:SparseUpper}
Under conditions~\ref{cond:design}, \ref{cond:asymp} and \ref{cond:sigma}, we further assume $k\in[p]$ and that~\eqref{Eq:klogpovern-1} holds. If $\rho^2\geq \frac{7\sigma^2 k\log p}{\underline{\lambda}^2n\kappa_1}$, and we set input parameters $\omega=\hat\sigma\sqrt{(4+\varepsilon)\log p}$ for any $\varepsilon\in(0,1]$ and $\tau\in(0, \hat\sigma^2 k\log p]$ in Algorithm~\ref{Algo:Test}, then
\[
\mathcal{M}_X(k,\rho) \leq \sup_{\beta\in\mathbb{R}^p} P_{\beta,\beta}^X(\psi_{\omega,\tau}^{\mathrm{sparse}}\neq 0) + \hspace{-0.3cm} \sup_{\substack{\beta_1,\beta_2\in\mathbb{R}^p\\ (\beta_1-\beta_2)/2\in\Theta_{p,k}(\rho)}} \hspace{-0.3cm} P_{\beta_1,\beta_2}^X(\psi_{\omega,\tau}^{\mathrm{sparse}} \neq 1) \xrightarrow{\mathrm{a.s.}} 0
\]
\end{cor}
Corollary~\ref{Cor:SparseUpper} shows that the test $\psi_{\omega,\tau}^{\mathrm{sparse}}$ has an asymptotic detection limit, measured in $\|\beta_1-\beta_2\|_2$, of at most $\{\frac{7\sigma^2 k\log p}{\underline{\lambda}^2 n\kappa_1}\}^{1/2}$ for all $k$ satisfying~\eqref{Eq:klogpovern-1}. While \eqref{Eq:klogpovern-1} is satisfied for $k\leq p^{\alpha}$ with any $\alpha\in[0,1)$, the detection limit upper bound shown in Corollary~\ref{Cor:SparseUpper} is suboptimal when $\alpha > 1/2$, as we will see later in Theorem~\ref{Thm:AlternativeDense}. On the other hand, Theorem~\ref{Thm:LowerBound} below shows that  when $\alpha < 1/2$, the detection limit of $\psi_{\omega,\tau}^{\mathrm{sparse}}$ is essentially optimal for a large subclass of covariance matrices. For some $D > 0$ (which we allow to diverge as $p\to\infty$), we write $\mathrm{RowSp}(D) \subseteq \mathbb{R}^{p\times p}$ for the subset of $p\times p$ matrices having at most $D$ nonzero elements in each row and define
\begin{equation}
\label{Eq:CD}
\mathcal{C}(D):=\biggl\{\Sigma \in \mathcal{C}: \text{$\Sigma = \Sigma_0 + \Gamma$ for $\Sigma_0 \in \mathrm{RowSp}(D)$ and $\|\Gamma\|_{\max}\leq \frac{D}{k\log^2 p}$}\biggr\}.
\end{equation}
The class $\mathcal{C}(D)$ consists of matrices admitting a sparse plus noise decomposition, and contains many common covariance matrices for relatively small choice of $D$. For instance, if $\Sigma$ is a banded matrix, we may take $D$ to be its bandwidth and $\Gamma = 0$. When $\Sigma = (\Sigma_{j,\ell})_{j,\ell\in[p]} = (\varrho^{|j-\ell|})_{j-\ell\in[p]}$ has an auto-regressive structure, we may take $D=(\log k + \log\log^2 p) / \log (varrho)$. Another example is when $\Sigma = V\Lambda V^\top + \Xi$ has a spiked covariance structure such that $V$ is uniformly sampled from $\mathbb{O}^{p\times r}$ and $\Lambda, \Xi\succeq 0$ are diagonal (this is commonly encountered in e.g.\ factor analysis). In this case, each row of $V$ has $\ell_2$ norm bounded by $\sqrt{(r\log p)/p}$ with high probability, so $\|V\Lambda V^\top\|_{\max} \leq (\overline{\lambda} r \log p)/p$ and hence $\Sigma \in \mathcal{C}(D)$ with $D = \max\{1, (\overline{\lambda} r \log^3 p)/p\}$. 

\begin{thm}
\label{Thm:LowerBound}
Under conditions~\ref{cond:design} and~\ref{cond:asymp}, if further assume $\Sigma \in\mathcal{C}(D)$ for some $D>0$, $k\leq p^{\alpha}$ for some $\alpha\in[0,1/2)$ and $\rho^2\leq \frac{(1-2\alpha-\varepsilon)\sigma^2 k\log p}{8Dn\kappa_1}$ for some $\varepsilon \in (0, 1-2\alpha]$, then $\mathcal{M}_X(k,\rho) \xrightarrow{\mathrm{a.s.}} 1$.
\end{thm}
For any fixed $\alpha < 1/2$, Theorem~\ref{Thm:LowerBound} shows that for designs having covariance matrix in $\mathcal{C}(D)$, if the squared signal $\ell_2$ norm is a factor of $56D/\{\underline{\lambda}^2(1-2\alpha-\varepsilon)\}$ smaller than what can be detected by $\psi_{\omega,\tau}^{\mathrm{sparse}}$ shown in Corollary~\ref{Cor:SparseUpper}, then all tests are asymptotically powerless in differentiating the null from the alternative. In other words, in the sparse regime where $k\leq p^{\alpha}$ for $\alpha < 1/2$, the test $\psi_{\omega,\tau}^{\mathrm{sparse}}$ has a minimax optimal detection limit measured in $\|\beta_1-\beta_2\|_2$, up to constants depending on $\alpha, \underline{\lambda}$ and $D$ only.

It is illuminating to relate the above results with the corresponding ones in the one-sample problem in the sparse regime ($\alpha < 1/2$).  Let $X$ be an $n\times p$ matrix with independent $N(0,1)$ entries and $Y=X\beta+\epsilon$ for $\epsilon\mid X \sim N(0, I_n)$, and we consider the one-sample problem to test $H_0: \beta=0$ against $H_1: \beta \in \Theta_{p,k}(\rho)$. Theorem~2 and 4 of \citet{AriasCastroetal2011} state that under the additional assumption that all nonzero entries of $\beta$ have equal absolute values, the detection limit for the one-sample problem is at $\rho \asymp \sqrt{\frac{k\log p}{n}}$, up to constants depending on $\alpha$.
Thus, when $\underline{\lambda}$ and $D$ are constants, Corollary~\ref{Cor:SparseUpper} and Theorem~\ref{Thm:LowerBound} suggest that the two-sample problem with model~\eqref{Eq:Model} has up to multiplicative constants the same detection limit as the one-sample problem with sample size 
\begin{equation}
\label{Eq:EffectiveSampleSize}
n\kappa_1 = \frac{nr}{(1+r)^2(1+s)},
\end{equation}
which unveils how this `effective sample size' depends on the relative proportions between sample sizes $n_1$, $n_2$ and the dimension $p$ of the problem. It is to be expected that the effective sample size is proportional to $m$, which is the number of observations constructed from $X_1$ and $X_2$ in $W$. More intriguingly,~\eqref{Eq:EffectiveSampleSize} also gives a precise characterization of how the effective sample size depends on the imbalance between the number of observations in $X_1$ and $X_2$. For a fixed $n=n_1+n_2$, $n\kappa_1$ is maximized when $n_1=n_2$ and converges to $n_1m/n$ (or $n_2m/n$) if $n_1/n\to 0$ (or $n_2/n\to 0$).

%When $\alpha \nearrow 1/2$, at the interface between the dense and the sparse regimes, the lower bound given in Theorem~\ref{Thm:LowerBound} becomes trivial. This phenomenon has been observed previously in the study of one-sample testing of high-dimensional regression coefficients \citep{AriasCastroetal2011,donoho2004higher}.
%The sharp detection boundary established in \cite{donoho2004higher}, after being translated to our terminology, is about $(\alpha -1/2)k\log p /n$ for $\alpha \in [1/4, 1/2)$, which also vanishes as $\alpha$ approaches $1/2$.
% another candidate sentence
%\cite{AriasCastroetal2011} consider the following setting: under the standard setting $Y = X \beta + \epsilon$, we wish to test $H_0: \beta = 0$ versus the alternative $H_1:$ (the alternative in \cite{AriasCastro2011} is a bit complex to summarise).
%The sharp detection boundary established in \cite{AriasCastroetal2011}, after being translated to our terminology, is about $(\alpha -1/2)k\log p /n$ for $\alpha \in [1/4, 1/2)$, which also vanishes as $\alpha$ approaches $1/2$.
%

\subsection{Dense case}

We now turn our attention to our second test, $\psi^{\mathrm{dense}}_{\eta}$. The following theorem states a sufficient signal $\ell_2$ norm size for which $\psi^{\mathrm{dense}}_{\eta}$ is asymptotically powerful in distinguishing the null from the alternative.

\begin{thm}
\label{Thm:AlternativeDense}
Under Conditions~\ref{cond:design}, \ref{cond:asymp} and \ref{cond:sigma-strong}, we let $\eta=\hat{\sigma}^2\bigl(m+2\sqrt{(2+\varepsilon) m\log p}+2(1+\varepsilon)\log p\bigr)$ for any $\varepsilon \in (0,5)$.
We further assume $k\in[p]$, $\rho^2\geq \frac{2\sigma^2\sqrt{m\log p}}{n\kappa_1\underline{\lambda}}$ and that \eqref{Eq:klogpovern-1} is satisfied.
\begin{enumerate}[label={\textup{(\alph*)}}]
\item If $\beta_1=\beta_2$, then $\psi^{\mathrm{dense}}_\eta(X_1,X_2,Y_1,Y_2)\xrightarrow{\mathrm{a.s.}} 0$.
\item If $\theta = (\beta_1-\beta_2)/2 \in \Theta_{p,k}(\rho)$, then $\psi^{\mathrm{dense}}_\eta(X_1,X_2,Y_1,Y_2)\xrightarrow{\mathrm{a.s.}} 1$.
\end{enumerate}
Consequently, $\mathcal{M}_X(k,\rho)\xrightarrow{\mathrm{a.s.}} 0$.
\end{thm} 

Theorem~\ref{Thm:AlternativeDense} indicates that the sufficient signal $\ell_2$ norm for asymptotic powerful testing via $\psi^{\mathrm{dense}}_{\eta}$ does not depend upon the sparsity level. While the above result is valid for all $k\in[p]$ such that~\eqref{Eq:klogpovern-1} holds, it is more interesting in the dense regime where $k\geq p^{1/2}$.
More precisely, by comparing Theorems~\ref{Thm:AlternativeDense} and~\ref{Cor:SparseUpper}, we see that if $k^2\log p > m$ and $k\log(ep/k)\leq n/(2C_{s,r})$, the test $\psi^{\mathrm{dense}}_{\eta}$ has a smaller provable detection limit than $\psi^{\mathrm{sparse}}_{\omega,\tau}$. In our asymptotic regime \ref{cond:asymp}, $m\asymp n \asymp p$, so $\frac{2\sqrt{m\log p}}{n\kappa_1}$ is, up to constants depending on $s$ and $r$, of order $p^{-1/2}\log^{1/2}p$. %We note that in our asymptotic setting condition~\eqref{Eq:klogpovern-1} restricts the sparsity $k$ to be of order $\smallO(p \log^{-1} p)$. In fact, as discussed after Proposition~\ref{Prop:WtW-general}, this condition can be removed, thus allowing the theorem to be applied for all $k\in[p]$, if the complementary-sketching transformed design matrix $W$ satisfies~\eqref{Eq:WtWdiag-1} and~\eqref{Eq:kopNorm-1}. 
The following theorem points out that when $\Sigma\in\mathcal{C}(D)$ for some constant $D$, the detection limit of $\psi^{\mathrm{dense}}_{\eta}$ is minimax optimal up to poly-logarithmic factors in the dense regime.
\begin{thm}
\label{Thm:LowerBoundDense}
%Under conditions~\ref{cond:design} and~\ref{cond:asymp}, if we further assume $p^{1/2}\leq k \leq p^{\alpha}$ for some $\alpha\in[1/2,1]$ and $\rho=o(p^{-1/4}\log^{-3/4} p)$, then $\mathcal{M}_X \xrightarrow{\mathrm{a.s.}} 1$.
Under conditions~\ref{cond:design} and~\ref{cond:asymp}, if we further assume $\Sigma \in\mathcal{C}(D)$ for some $D>0$, $p^{1/2}\leq k \leq p^{\alpha}$ for some $\alpha\in[1/2,1)$ and $\rho^2=\smallO(p^{-1/2}\min\{\log^{-1/2} (ep/k), D^{-3/2}\} )$, then $\mathcal{M}_X(k,\rho) \xrightarrow{\mathrm{a.s.}} 1$.
\end{thm}

%Theorem~\ref{Thm:LowerBoundDense} points out that the lower bound on detectable signal size $\rho^2$ prescribed in Theorem~\ref{Thm:AlternativeDense} is necessary up to poly-logarithmic factors. 

%We remark that, the precise lower limit $1/2$ for $p/\max\{n_1,n_2\}$ here is not important and can be replaced by any positive constant in $(0,1)$.
%\fn{What is the difference between the current version + remark and statement with $p/n_1, p/n_2 \in [\delta,1)$ for some $\delta\in (0,1)$?}

%!TEX root = ../compsket.tex

\section{Numerical studies}
\label{sec:simu}
In this section, we study the finite sample performance of our proposed procedures via numerical experiments. Unless otherwise stated, the data generating mechanism for all simulations in this section is as follows.
We first generate design matrices $X_1$ and $X_2$ with independent $N(0,1)$ entries. Then, for a given sparsity level $k$ and a signal strength $\rho$, set $\Delta = (\Delta_j)_{j\in[p]}$ so that $(\Delta_1,\ldots,\Delta_k)^\top \sim \rho\mathrm{Unif}(\mathcal{S}^{k-1})$ and $\Delta_j=0$ for $j > k$. We then draw $\beta_1\sim N_p(0,I_p)$ and define $\beta_2:=\beta_1+\Delta$. Finally, we generate $Y_1$ and $Y_2$ as in~\eqref{Eq:Model}, with $\epsilon_1\sim N_{n_1}(0,I_{n_1})$ and $\epsilon_2\sim N_{n_2}(0, I_{n_2})$ independent of each other.

In Section~\ref{Sec:ValidateTheory}, we supply the oracle value of $\hat\sigma^2=1$ to our procedures to check whether their finite sample performance is in accordance with our theory.
In all subsequent subsections where we compare our methods against other procedures, we estimate the noise variance $\sigma^2$ with the method-of-moments estimator proposed by \citet{Dicker2014}.
%Specifically, after obtaining $W$ and $Z$ in Step 4 of Algorithm~1, we compute
%\[
%\hat M_1 := \frac{1}{p}\mathrm{tr}\biggl(\frac{1}{m}W^\top W\biggr) \quad \text{and}\quad \hat M_2:=\frac{1}{p}\mathrm{tr}\biggl\{\biggl(\frac{1}{m}W^\top W\biggr)^2\biggr\} - \frac{1}{pm}\biggl\{\mathrm{tr}\biggl(\frac{1}{m}W^\top W\biggr)^2\biggr\}.
%\]
%This allows us to estimate
%\[
%\hat\sigma^2 := \biggl\{1+\frac{p\hat M_1^2}{(m+1)\hat M_2^2}\biggr\}\frac{\|Z\|_2^2}{m} - \frac{\hat M_1}{m(m+1)\hat M_2}\|W^\top Z\|_2^2.
%\]
We implement our estimators $\psi^{\mathrm{sparse}}_{\omega,\tau}$ and $\psi^{\mathrm{dense}}_{\eta}$ on standardized data $X_1/\hat\sigma$, $X_2/\hat\sigma$, $Y_1/\hat\sigma$ and $Y_2/\hat\sigma$ with the tuning parameters $\omega=2\hat\sigma \sqrt{\log p}$, $\tau = \hat\sigma^2\log p$ and $\eta=\hat\sigma^2(m+\sqrt{8m\log p}+4\log p)$ as suggested by Theorems~\ref{Thm:Null},  \ref{Thm:Alternative} and~\ref{Thm:AlternativeDense}.

\subsection{Effective sample size in two-sample testing}
\label{Sec:ValidateTheory}
We first investigate how the empirical power of our test $\mathrm{\psi}_{\lambda,\tau}^{\mathrm{sparse}}$ relies on various problem parameters.
In light of our results in Theorems~\ref{Thm:Null} and~\ref{Thm:Alternative}, we define
\begin{equation}
\label{Eq:nu}
\nu: = \frac{rn\rho^2}{\sigma^2(1+s)(1+r)^2k\log p},
\end{equation}
where $s := p/m$ and $r:=n_1/n_2$. Note that in the asymptotic regime (C2), we have $\nu \to n\kappa_1\rho^2 / (\sigma^2k\log p)$. As discussed after Theorem~\ref{Thm:Alternative}, $rn/\{(1+s)(1+r)^2\}$ in the definition of $\nu$ is asymptotically $n\kappa_1$ and can be viewed as the effective sample size in the testing problem.
In Figure~\ref{Fig:PhaseTransition}, we plot the estimated test power of $\psi_{\lambda,\tau}^{\mathrm{sparse}}$ against $\nu$ over 100 Monte Carlo repetitions for $n=1000$, $k=10$, $\rho\in\{0,0.2,\ldots,2\}$ and various values of $p$ and $n_1$.
In the left panel of Figure~\ref{Fig:PhaseTransition},  $p$ ranges from $100$ to $900$, which corresponds to $s$ from $1/9$ to $9$. As for the right panel, we vary $n_1$ from $100$ to $900$, which corresponds with an $r$ varying between $1/9$ and $9$.
In both panels, the power curves for different $s$ and $r$ values overlap each other, with the phase transition all occurring at around $\nu\approx 1.5$. This conforms well with the effective sample size and the detection limit articulated in our theory.
%\fn{Sharper transition from 0 to 1 if $p\nearrow$?}

\begin{figure}[htbp]
\centering
\begin{tabular}{cc}
\includegraphics[width=0.48\textwidth]{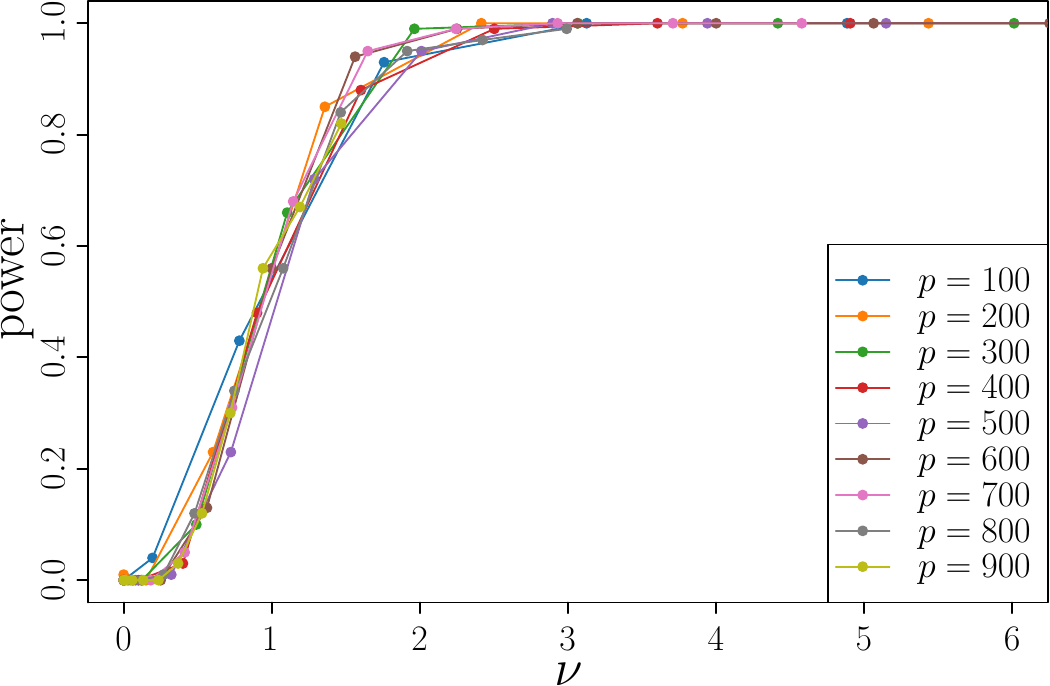} &
\includegraphics[width=0.48\textwidth]{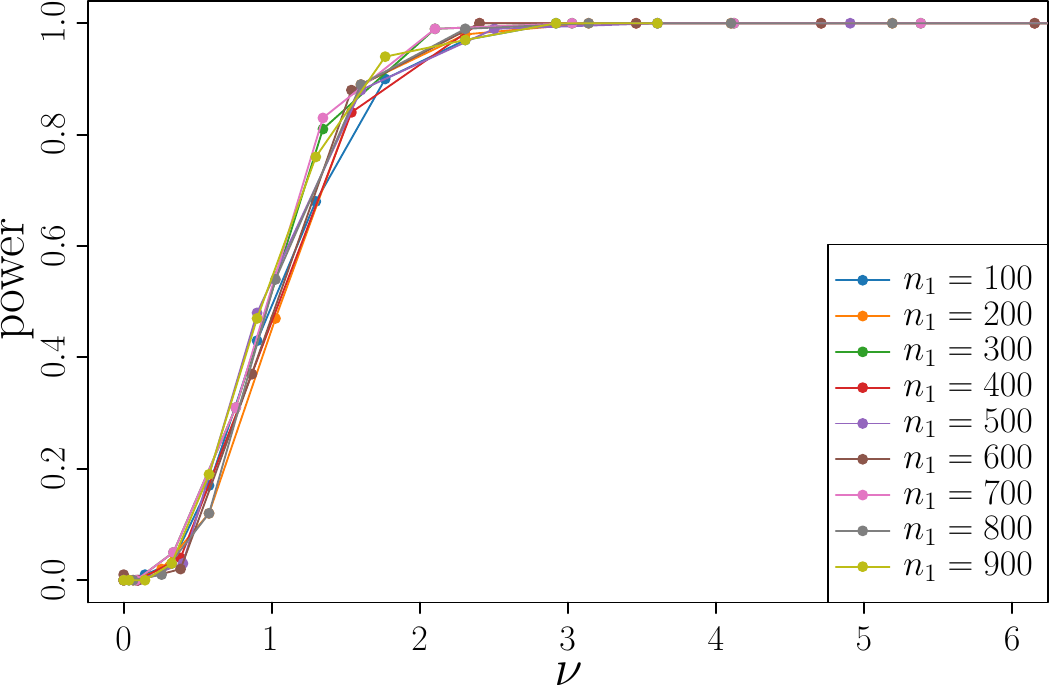}
\end{tabular}
\caption{\label{Fig:PhaseTransition}Power function of $\psi_{\lambda,\tau}^{\mathrm{sparse}}$, estimated over 100 Monte Carlo repetitions, plotted against $\nu$, as defined in~\eqref{Eq:nu}, in various parameter settings. Left panel: $n_1=n_2=500$, $p\in\{100,200,\ldots,900\}$, $k=10$, $\rho\in\{0,0.2,\ldots,2\}$. Right panel: $n_1\in\{100,200,\ldots,900\}$, $n_2=1000-n_1$, $p=400$, $k=10$, $\rho\in\{0,0.2,\ldots,2\}$.}
\end{figure}

\subsection{Comparison with other methods}
\label{Sec:Comparison}
Next, we compare the performance of our procedures against competitors in the existing literature. The only methods we were aware of that could allow for dense regression coefficients $\beta_1$ and $\beta_2$ were those proposed by \citet{ZhuBradic2016} and \citet{Chabonnier2015}. In addition, we also include in our comparisons the classical likelihood ratio test, denoted by $\psi^{\mathrm{LRT}}$, which rejects the null when the $F$-statistic defined in~\eqref{Eq:LRT} exceeds the upper $\alpha$-quantile of an $F_{p,\, n-2p}$ distribution. Note that the likelihood ratio test is only well-defined if $p<\min\{n_1,n_2\}$. The test proposed by \citet{ZhuBradic2016}, which we denote by $\psi^{\mathrm{ZB}}$, requires that $n_1=n_2$ (when the two samples do not have equal sample size, a subset of the larger sample would be discarded for the test to apply). Specifically, writing $X_+:=X_1+X_2$, $X_-:=X_1-X_2$ and $Y_+:=Y_1+Y_2$, $\psi^{\mathrm{ZB}}$ first estimates $\gamma=(\beta_1+\beta_2)/2$ and
\[
\Pi:=\{\mathbb{E}(X_+^\top X_+)\}^{-1}\mathbb{E}(X_+^\top X_-)
\]
by solving Dantzig-Selector-type optimization problems. Then based on the obtained estimators $\hat\gamma$ and $\hat\Pi$,  $\psi^{\mathrm{ZB}}$ proceeds to compute a test statistic
\[
T_{\mathrm{ZB}} := \frac{\|\{X_- - X_+\hat\Pi\}^\top \{Y_+ - X_+\hat\gamma\}\|_\infty}{\|Y_+ - X_+\hat\gamma\|_2}.
\]
\begin{figure}[htbp]
\centering
\begin{tabular}{c} \hspace{0.1\textwidth}\includegraphics[width=0.7\textwidth]{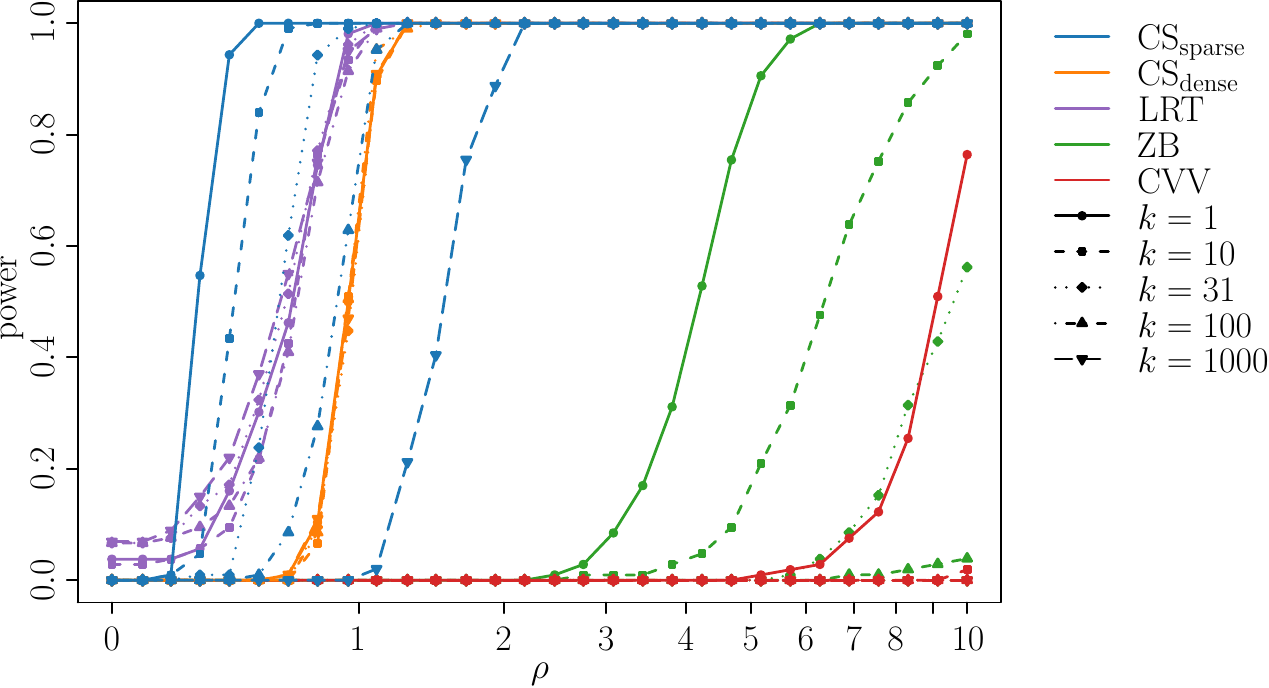}
\\
\\
\hspace{0.1\textwidth}\includegraphics[width=0.7\textwidth]{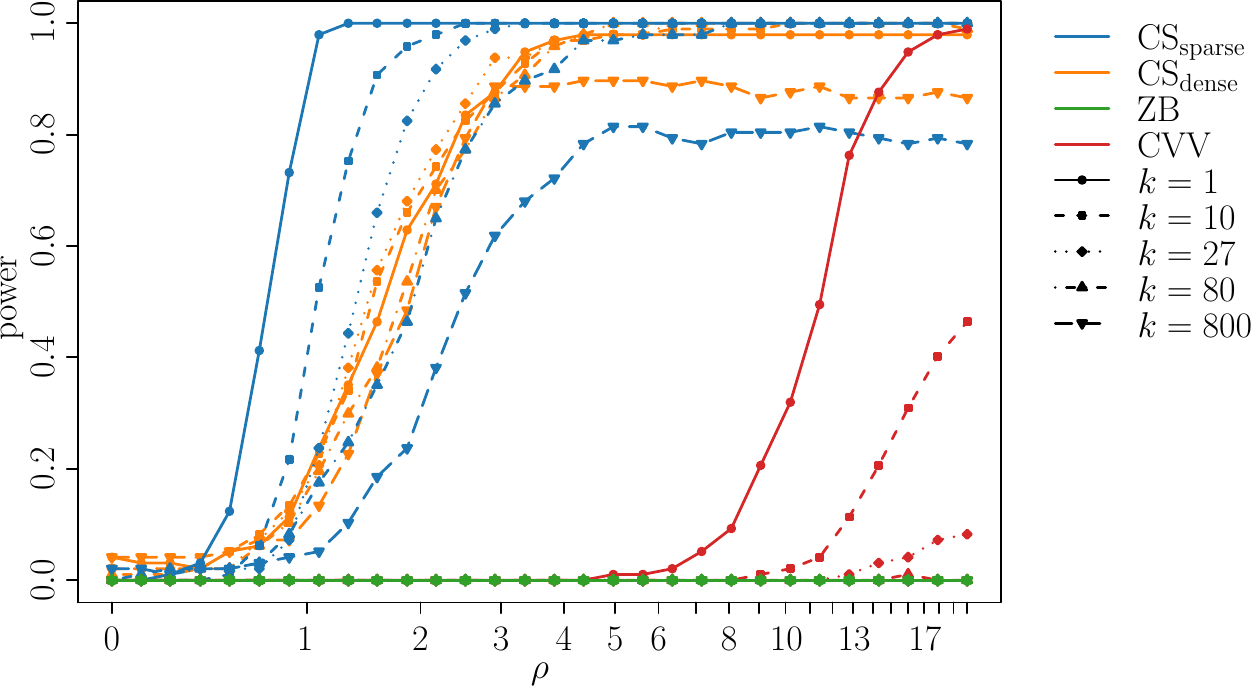} \end{tabular}
\caption{\label{Fig:comparison}Power comparison of different methods at different sparsity levels $k\in\{1,10,\lfloor p^{1/2}\rfloor,0.1p, p\}$ and different signal $\ell_2$ norm $\rho$ on a logarithmic grid (noise variance $\sigma^2=1$). Top panel: $n_1=n_2=1200$, $p=1000$, $\rho\in[0,10]$; bottom panel: $n_1=n_2=500$, $p=800$, $\rho \in [0,20]$.}
\end{figure}
Their test rejects the null if the test statistic exceeds an empirical upper-$\alpha$-quantile (obtained via Monte-Carlo simulation) of $\|\xi\|_\infty$ for $\xi\sim N(0, \{X_- - X_+\hat\Pi\}^\top \{X_- - X_+\hat\Pi\})$. As the estimation of $\Pi$ involves solving a sequence of $p$ Dantzig Selector problems, which is often time-consuming, we have implemented $\psi^{\mathrm{ZB}}$ with the oracle choice of $\hat\Pi = \Pi$, which is equal to $I_p$ when covariates in the two design matrices $X_1$ and $X_2$ follow independent centred distribution with the same covariance matrix. The test proposed by \citet{Chabonnier2015}, denoted here by $\psi^{\mathrm{CVV}}$, first performs a LARS regression \citep{Efronetal2004} of concatenated response $Y = (Y_1^\top, Y_2^\top)^\top$ against the block design matrix
\[
\begin{pmatrix}
X_1 & X_1\\ X_2 & -X_2
\end{pmatrix}
\]
to obtain a sequence of regression coefficients $\hat b = (\hat b_1, \hat b_2)\in\mathbb{R}^{p+p}$. Then for every $\hat b$ on the LARS solution path with $\|\hat b\|_0\leq \min\{n_1,n_2\}/2$, they restrict the original testing problem into the subset of coordinates where either $\hat b_1$ or $\hat b_2$ is non-zero, and form test statistics based on the Kullback--Leibler divergence between the two samples restricted to these coordinates. The sequence of test statistics is then compared with Bonferonni-corrected thresholds at size $\alpha$. For both the $\psi^{\mathrm{LRT}}$ and $\psi^{\mathrm{CVV}}$, we set $\alpha=0.05$.

Figure~\ref{Fig:comparison} compares the estimated power, as a function of $\|\beta_1-\beta_2\|_2$, of $\psi^{\mathrm{sparse}}_{\omega,\tau}$ and $\psi_{\eta}^{\mathrm{dense}}$ against that of $\psi^{\mathrm{LRT}}$, $\psi^{\mathrm{ZB}}$ and $\psi^{\mathrm{CVV}}$. We ran all methods on the same 100 datasets for each set of parameters. We performed numerical experiments in two high-dimensional settings with different sample-size-to-dimension ratio: $p=1000$, $n_1=n_2=1200$ in the left panel and $p=800$, $n_1=n_2=500$ in the right panel. Here, we took $n_1 = n_2$ to maximize the power of $\psi^{\mathrm{ZB}}$. Also, since the likelihood ratio test requires $p<\min\{n_1,n_2\}$, it is only implemented in the left panel. For each experiment, we varied $k$ in the set $\{1,10,\lfloor p^{1/2}\rfloor, 0.1p, p\}$ to examine different sparsity levels.

We see in Figure~\ref{Fig:comparison} that both $\psi_{\lambda,\tau}^{\mathrm{sparse}}$ and $\psi_{\eta}^{\mathrm{dense}}$ showed promising finite sample performance. Both our tests did not produce any false positives under the null when $\rho=0$, and showed better power compared to $\psi^{\mathrm{ZB}}$ and $\psi^{\mathrm{CVV}}$. In the more challenging setting of the right panel with $p>\max\{n_1,n_2\}$, it takes a signal $\ell_2$ norm more than 10 times smaller than that of the competitors for our test $\psi^{\mathrm{sparse}}_{\omega,\tau}$ to reach power of almost $1$ in the sparsest case.
Note though, in the densest case on the right panel ($k=800$), $\psi^{\mathrm{sparse}}_{\omega,\tau}$ and $\psi^{\mathrm{dense}}_{\eta}$ did not have saturated power curves, because noise variance is over-estimated by $\hat\sigma^2$ in this setting.
%\fn{Present the oracle version of our tests with known $\sigma =1$ and display along the non-oracle version?}

We also observe that the power of $\psi^{\mathrm{sparse}}_{\omega,\tau}$ has a stronger dependence on the level $k$ than that of $\psi^{\mathrm{dense}}_{\eta}$. For $k\leq \sqrt{p}$, $\psi^{\mathrm{sparse}}_{\omega,\tau}$ appears much more sensitive to the signal size. As $k$ increases, $\psi_{\eta}^{\mathrm{dense}}$ eventually outperforms $\psi_{\lambda,\tau}^{\mathrm{sparse}}$, which is consistent with our observed phase transition behaviour as discussed after Theorem~\ref{Thm:AlternativeDense}. It is interesting to note that when the likelihood ratio test is well-defined (left panel), it has better power than $\psi^{\mathrm{dense}}_{\eta}$. This is partly due to the fact that the theoretical choice of threshold $\eta$ is relatively conservative to ensure asymptotic size of the test is 0 almost surely. In comparison, the rejecting threshold for the likelihood ratio test is chosen to have ($p$ fixed and $n\to\infty$) asymptotic size of $\alpha=0.05$, and the empirical size is sometimes observed to be larger than $0.08$.

As remarked at the beginning of Section~\ref{Sec:OurContribution}, the complementary sketching transforming can potentially be combined with other one-sample global testing procedure to obtain a two-sample test. Figure~\ref{Fig:carp} illustrates this by comparing our methods with a two-stage procedure combining the complementary sketching transformation with the one-sample test proposed in \citet{carpentier2019minimax}, which we call $\psi^{\mathrm{CCCTW}}$. We remark that the test in \citet{carpentier2019minimax} requires the knowledge of the sparsity $k$ and involves an unspecified parameter $C_*$. In our experience, the optimal choice of $C_*$ seems to vary with different sparsity levels. As such, we have implemented $\psi^{\mathrm{CCCTW}}$ by choosing $C_*$ in each simulation setting to maximize the power subject to a size constraint of $\alpha=0.05$ (specifically, for $k=1,10,31,100,1000$, we have chosen $C_* = 0.30, 0.67, 1.56, 3.37,2.84$ respectively). We note that even granting $\psi^{\mathrm{CCCTW}}$ access to the additional sparsity parameter and this strong oracle parameter choice, $\psi^{\mathrm{sparse}}_{\omega,\tau}$ and $\psi^{\mathrm{dense}}_{\eta}$ are still competitive and in most cases outperforming $\psi^{\mathrm{CCCTW}}$ in the sparse and dense regimes respectively.  
We attribute this difference in performance to the fact that the matrix $W$ after complementary sketching transformation does not satisfy typical design conditions (such as independent rows) assumed in most one-sample testing literature. 
As a result, two-stage methods such as the $\psi^{\mathrm{CCCTW}}$ test may suffer from power loss due to model misspecification.

\begin{figure}[htbp]
\centering
\includegraphics[width=0.6\textwidth]{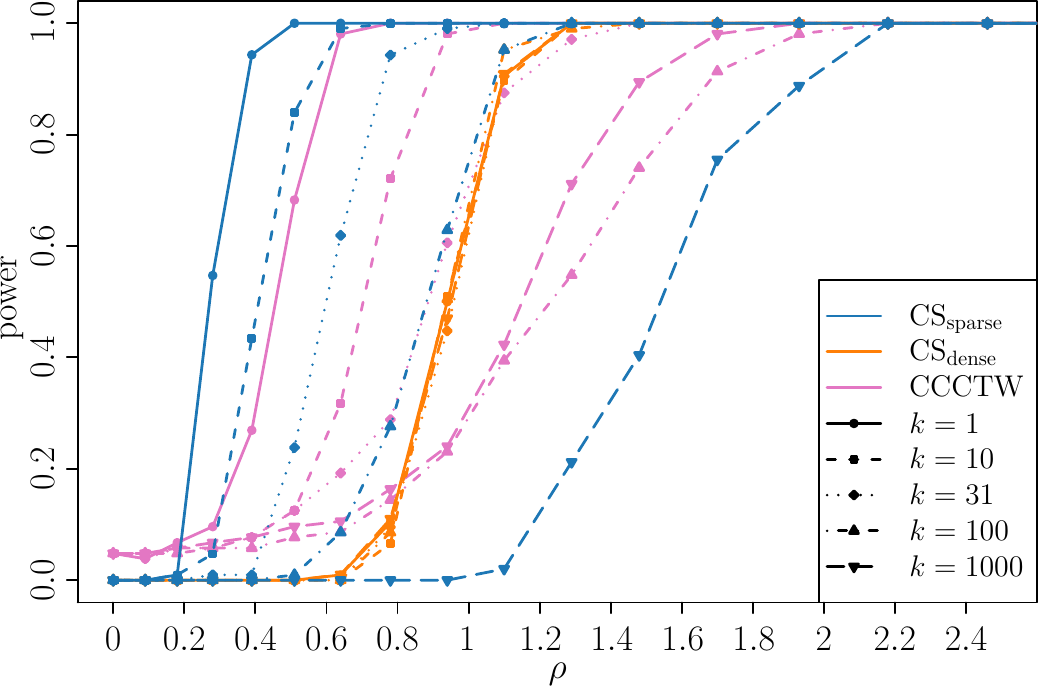}
\caption{\label{Fig:carp}Power comparison of methods constructed by using different one-sample testing procedures after complementary sketching transformation. Parameters: $n_1=n_2=1200$, $p=1000$, $\rho\in[0,2.5]$, $k\in\{1,10,31,100,1000\}$.}
\end{figure}

\subsection{More general data generating mechanisms}
\label{Sec:Misspecification}
We have thus far focused on the case of Gaussian random design $X_1, X_2$ with identity covariance and Gaussian regression noises $\epsilon_1,\epsilon_2$.  However, as our proposed testing procedures can still be used under more general data generating mechanisms. We consider the following four setups:
\begin{enumerate}[label={(\alph*)}]
\item Correlated design: assume rows of $X_1$ and $X_2$ are independently drawn from $N(0, \Sigma)$ with $\Sigma = (2^{-|j_1-j_2|})_{j_1,j_2\in[p]}$.
\item Rademacher design: assume entries of $X_1$ and $X_2$ are independent Rademacher random variables.
\item One way balanced ANOVA design: assume $d_1:=n_1/p$ and $d_2:=n_2/p$ are integers and $X_1$ and $X_2$ are block diagonal matrices
\[
X_1=\begin{pmatrix} \mathbf{1}_{d_1} & & \\ & \ddots & \\ & & \mathbf{1}_{d_1}\end{pmatrix} \qquad
X_2=\begin{pmatrix} \mathbf{1}_{d_2} & & \\ & \ddots & \\ & & \mathbf{1}_{d_2}\end{pmatrix},
\]
where $\mathbf{1}_d$ is an all-one vector in $\mathbb{R}^d$.
\item Heavy tailed noise: we generate both $\epsilon_1$ and $\epsilon_2$ with independent $t_4/\sqrt{2}$ entries. Note that the $\sqrt{2}$ denominator standardizes the noise to have unit variance, to ensure easier comparison between settings.
\end{enumerate}
In setups (a) to (c), we keep $\epsilon_1\sim N_{n_1}(0,I_{n_1})$ and $\epsilon_2\sim N_{n_2}(0,I_{n_2})$ and in setup (d), we keep $X_1$ and $X_2$ to have independent $N(0,1)$ entries. Note that in setup (a) the covariance matrix belongs to $\mathcal{C}(D)$ with $D\asymp \log k + \log\log^2 p$. Figure~\ref{Fig:misspecification} compares the performance of $\psi_{\lambda,\tau}^{\mathrm{sparse}}$, $\psi^{\mathrm{dense}}_{\eta}$ with that of $\psi^{\mathrm{ZB}}$ and $\psi^{\mathrm{CVV}}$. In all settings, we set $n_1=n_2=500$ and $k=10$. In settings (a), (b) and (d), we choose $p=800$ and $\rho$ from 0 to 20. In setting (c), we choose $p=250$ and $\rho$ from 0 to 50. We see that $\psi^{\mathrm{sparse}}_{\omega,\tau}$ is robust to model misspecification and exhibits good power in all settings. The test $\psi^{\mathrm{desne}}_{\eta}$ is robust to non-normal design and noise, but exhibits a slight reduction in power in a correlated design. The advantage of $\psi^{\mathrm{sparse}}_{\omega,\tau}$ and $\psi^{\mathrm{dense}}_{\eta}$ over competing methods is least significant in the ANOVA design in setting (c), where each row vector of the design matrices has all mass concentrated in one coordinate. In all other settings where the rows of the design matrices are more `incoherent' in the sense that all coordinate have similar magnitude, $\psi^{\mathrm{sparse}}_{\omega,\tau}$ and $\psi^{\mathrm{dense}}_{\eta}$ start having nontrivial power at a signal $\ell_2$ norm 10 to 20 times smaller than that of the competitors.

\begin{figure}[htbp]
\centering
    \begin{subfigure}[b]{0.48\textwidth}
        \includegraphics[width=\textwidth]{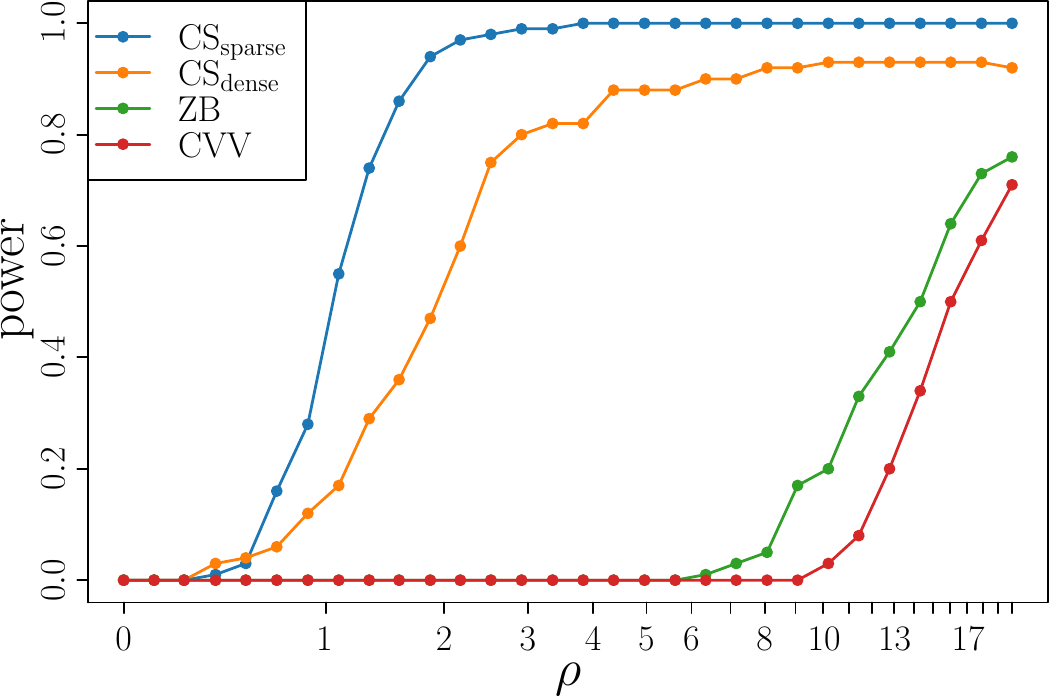}
        \caption{Correlated design $\Sigma=(2^{-|i-j|})_{i,j\in[p]}$}
        %\label{fig:3a}
    \end{subfigure}
    ~
    \begin{subfigure}[b]{0.48\textwidth}
        \includegraphics[width=\textwidth]{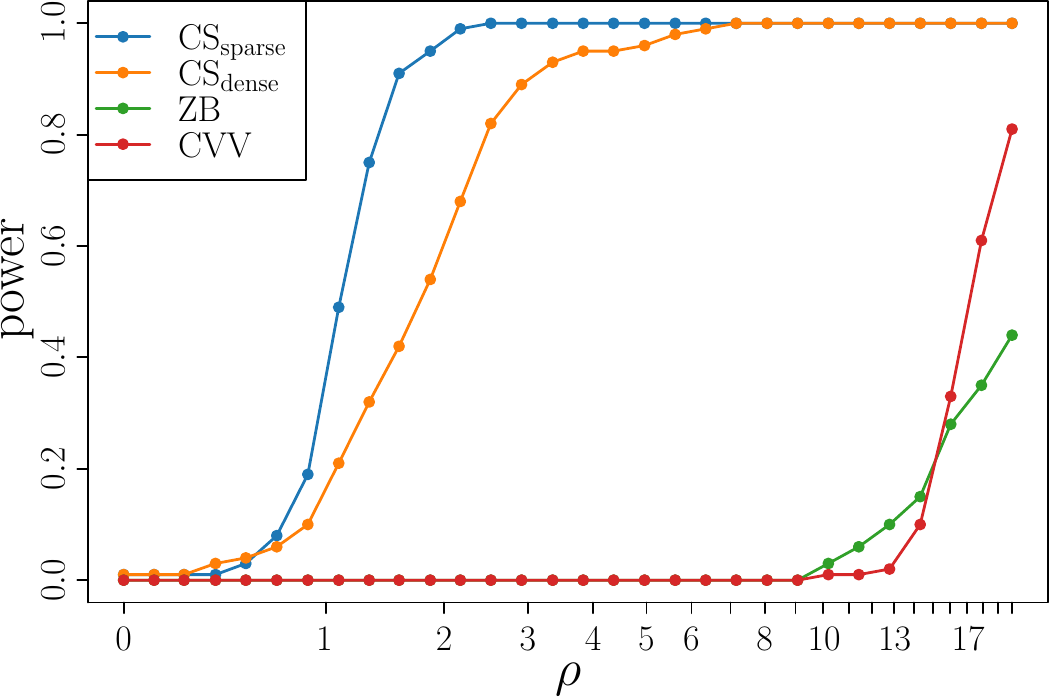}
        \caption{Rademacher design}
        %\label{fig:3b}
    \end{subfigure}
    \\[.2cm]
    \begin{subfigure}[b]{0.48\textwidth}
        \includegraphics[width=\textwidth]{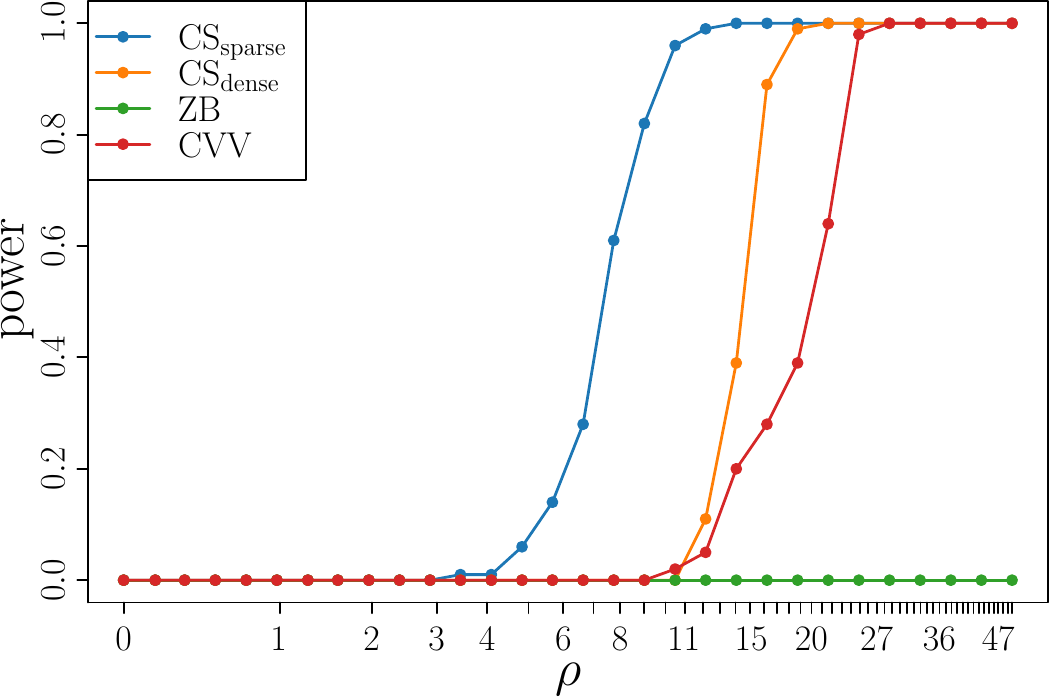}
        \caption{One-way balanced ANOVA design}
        %\label{fig:3c}
    \end{subfigure}
    ~
    \begin{subfigure}[b]{0.48\textwidth}
        \includegraphics[width=\textwidth]{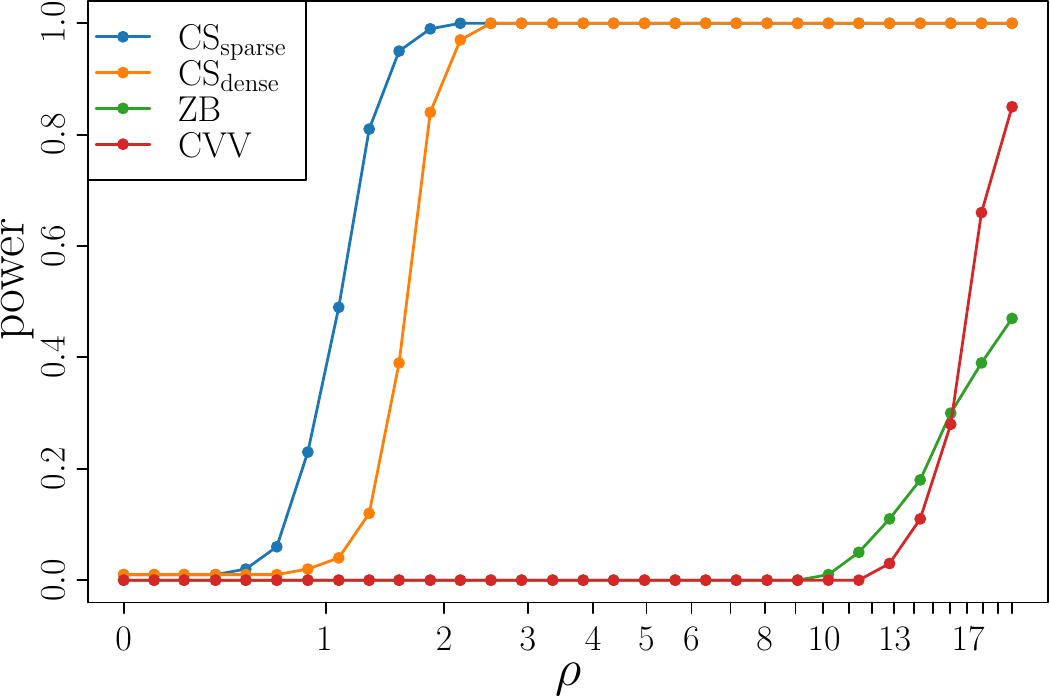}
        \caption{Gaussian design with $t_4/\sqrt{2}$ noise}
        %\label{fig:3d}
    \end{subfigure}
\caption{\label{Fig:misspecification}Power functions of different methods in models with non-Gaussian design or non-Gaussian noise, plotted against signal $\ell_2$ norm $\rho$ on a logarithmic grid.  Details of the models are in Section~\ref{Sec:Misspecification}.}
\end{figure}

%!TEX root = ../compsket.tex

\section{Analysing a single-cell dataset}
\label{sec:data}
Here, we illustrate the applicability of our methodology on a single-cell RNA sequencing dataset from \cite{Suo2022.01.17.476665}. The dataset consists of the logarithmic normalized gene expression levels of 33538 genes measured in 91298 cells. For simplicity, we focus on the subset of $n=7816$ cells that have been labelled as either CD4$^+$ T cells ($n_1= 4852$) or T regulatory (TREG) cells ($n_2=2964$), two closely related T cell subpopulations, and only keep the $p=4123$ genes whose log normalized expression variance is at least 1 in the two cell subpopulations. We are interested in testing for difference in the gene regulatory networks in the two cell subpopulations. This can be modelled by the difference in their respective Gaussian graphical model networks and tested by comparing the nodewise regression coefficients of each gene against the remaining genes in CD4$^+$ T cells and TREG cells. The left column of Table~\ref{Tab:RealData} summarizes the genes that report significant difference in their nodewise regression coefficients from our complementary sketching method, which we call `master regulators'. Among the nine genes identified to have significant difference in their nodewise regression coefficients, FOXP3, CTLA4, IL2RA, IL7R, IKZF2, CD83, ANXA1 are all known to be important regulators, from several independent pathways, essential for the function of the TREG cell type \citep{bayer2007function,walker2013treg,kim2015stable,doebbeler2018cd83, toomer2019essential, bai2020targeting}. 

In addition to identifying the master regulator genes, a slight modification of our algorithm also allows us to identify their top interacting partners insofar as the two T cell subpopulations are concerned. Specifically, after performing complementary sketching to obtain sketched design $W$ and response $Z$ in Step 4 of Algorithm~\ref{Algo:Test}, we may compute the Lasso solution path \citep{tibshirani1996regression}. The right column of Table~\ref{Tab:RealData} shows genes corresponding to the first eight nonzero coefficients entering the solution path, which can be interpreted as the top interacting partners of the master regulator genes. 

Computationally, we remark that when applying Algorithm~\ref{Algo:Test} to a differential network analysis setting, we can precompute an orthonormal basis spanning the orthogonal complement of the column span of $(X_1^\top, X_2^\top)^\top$, and obtain individual sketching matrices $A$ for each nodewise regression by augmenting that basis with one additional vector. For example, on an 8-core 3.20 GHz desktop machine, our algorithm was able to test for all $p=4123$ pairs of nodewise regressions in 1.6 hours (averaging 1.4 seconds per node).
%; in contrast, the methods of \citet{Chabonnier2015} and \citet{ZhuBradic2016} each takes more than 15 hours to run a nodewise regression test for just one single gene). 
Our code and preprocessed dataset for the real data analysis are both available on GitHub\footnote{\url{https://github.com/wangtengyao/compsket/}}.

It is interesting to contrast our analysis to the common differential-expression-based approach for identifying master regulator genes which determine the identity of different cell types. Differential expression analysis simply compares the expression levels of a gene in two different cell types, typically with the Mann--Whitney--Wilcoxon test. We have highlighted in bold in Table~\ref{Tab:RealData} all genes that are differentially expressed in CD4$^+$ T cells and TREG cells (at 0.05 level after Bonferroni correction). It can be seen that all our master regulators are differentially expressed. However, differential expression analysis identifies a much larger set of genes, many potentially belonging to the same pathway and dependent on each other. Overall, our complementary sketching approach allows for more precise identification of the central players in gene regulatory networks.
\begin{table}[htp]
\begin{center}
% {\footnotesize
\begin{tabular}{l p{0.7\textwidth}}
    \toprule
master regulators & top interacting partners\\
\midrule
\textbf{IKZF2} & MT-ND4L, \textbf{HLA-B}, MT-ATP8, \textbf{ETS1}, \textbf{FYB1}, \textbf{JUNB}, RNF213, \textbf{HLA-C}\\
\textbf{FOXP3} & MT-ND4L, MT-ATP8, \textbf{S100A4}, \textbf{CD96}, \textbf{ISG20}, BIRC2, \textbf{SRSF4}, \textbf{GZMM}\\
\textbf{CD83} & \textbf{HSPA1A}, \textbf{NFKBIA}, \textbf{RGS2}, \textbf{MTRNR2L12}, \textbf{NR4A1}, \textbf{PSMC3}, BAG3, \textbf{SRP9}\\
\textbf{IL2RA} & \textbf{ENO1}, \textbf{ARID5B}, RPL23, CDC42, \textbf{CREM}, \textbf{CISH}, \textbf{GADD45B}, \textbf{PMAIP1}\\
\textbf{ANXA1} & \textbf{JUNB}, \textbf{JUN}, \textbf{TNFAIP3}, \textbf{FOSB}, \textbf{CALM2}, \textbf{ABLIM1}, \textbf{RGS2}, \textbf{CHI3L2}\\
\textbf{CD8A} & \textbf{FTL}, \textbf{SLC25A3}, \textbf{CD8B}, COTL1, \textbf{PTPRCAP}, PCBP1, \textbf{STMN1}, IGFBP2\\
\textbf{CTLA4} & \textbf{RGS1}, \textbf{GBP2}, \textbf{RPS10}, ZFP36L1, \textbf{TAGAP}, \textbf{STAT3}, \textbf{RPS4Y1}, \textbf{SRGN}\\
\textbf{GNG8} & RPL41, HSPB1, \textbf{OST4}, LTB, \textbf{TERF2IP}, \textbf{CUTA}, \textbf{PPDPF}, \textbf{IFITM1}\\
\textbf{IL7R} & RPL41, \textbf{RPL27A}, \textbf{VIM}, \textbf{TRBC2}, SLC25A6, \textbf{CORO1A}, \textbf{RPS26}, \textbf{TRAC}\\
\bottomrule
\end{tabular}
%}
\end{center}
\caption{\label{Tab:RealData}Genes with significant difference identified by the complementary sketching algorithm, together with their top eight interacting partners using graphical Lasso post complementary sketching. Genes that are identified to be significant by the Mann--Whitney--Wilcoxon test after Bonferroni correction are shown in bold.}
\end{table}

% All on the gene CTL4A on Fengnan's home desktop
% ZhuBradic: 57307.589 sec elapsed
% Chabonnier: 60094.713 sec elapsed
% Matlab: 
% tic;
% [test_stat, test_result] = complementarySketching(X1,X2,y1,y2);
% toc;
% Elapsed time is 6.912575 seconds. 

%% differential network analysis on all nodes
% 16.04.2022 on Matlab with parfor on Fengnan's home desktop
% Elapsed time is 8901.249493 seconds. (2.47 hours)
% 18.04.2022 on Matlab with parfor on Fengnan's home desktop
% Elapsed time is 5873.346849 seconds.
% 17.04.2022 on Python on Fengnan's home desktop
% 9115.7 seconds

%% differential network analysis on all nodes
% 17.04.2022 on Matlab with parfor on Fengnan's computing server
% Elapsed time is 20170.922241 seconds. (5.6 hours)
% 17.04.2022 on Python on Fengnan's computing server
% 10579 seconds
% 18.04.2022 on Matlab with for on Fengnan's computing server
% 

%!TEX root = ../compsket.tex

\section{Proof of main results}
\label{sec:proofs}
\begin{proof}[Proof of Theorem~\ref{Thm:Null}]
By Condition~\ref{cond:sigma}, we may work on the almost sure event $\Omega_\sigma := \{ \hat{\sigma}/ \sigma = 1 + \smallO(1)\}$.
Under the null hypothesis where $\beta_1= \beta_2$, we have $\theta = 0$ and therefore, $Z = W\theta + \xi = \xi \sim N_m(0, \sigma^2 I_m)$. In particular, $Q_j/\sigma \sim N(0,1)$ for all $j \in [p]$.

Thus, noting the independence of $\hat{\sigma}$ and the sample and employing a union bound, we have for $\omega=\hat\sigma\sqrt{(4+\varepsilon)\log p}$ and any $\tau > 0$ that
\begin{equation*}
    \mathbb{P}(T\geq \tau \mid \hat{\sigma} ) \leq \sum_{j=1}^p \mathbb{P}(|Q_j| /\sigma \geq \omega/\sigma \mid \hat{\sigma}) \leq p \exp\bigl({-\omega^2/(2\sigma^2})\bigr). % = p^{-1-\varepsilon/2},
\end{equation*}
Keeping the preceding display in mind, by the independence of $\hat{\sigma} $ and the sample, we bound for $p$ sufficiently large
\begin{align*}
    \mathbb{P}( T \ge \tau \mid \Omega_\sigma) & \le p \exp\bigl( - (1 - \smallO(1))(2+\varepsilon/2) \log p\bigr) \\
    & \leq p \exp( - (2 + \varepsilon/4) \log p).
\end{align*}
Noting that $\Omega_\sigma$ is an almost sure event and that $p^{-1-\varepsilon/4}$ is summable for any $\varepsilon>0$, the almost sure convergence in the theorem statement follows from the Borel--Cantelli lemma.
\end{proof}

To prove Proposition~\ref{Prop:WtW-general}, we need the following proposition, which considers tail bounds for $\|Wu\|$ for a fixed $u\in\mathcal{S}^{p-1}$ in the special case $\Sigma = I_p$.
\begin{prop}
    \label{Prop:WtW}
    Under the conditions of Proposition~\ref{Prop:WtW-general}, we further assume $\Sigma=I_p$. There exists a random sequence $(h_n)_n$ such that $h_n \xrightarrow{\mathrm{a.s.}} 4\kappa_1$ for $\kappa_1$ defined in Lemma~\ref{lemma:spectral-limit} and that for any sequence $(\delta_n)_n$ satisfying $\log(1/\delta_n) = \smallO(p)$, we have for all large $p$ that
    \[
    \mathbb{P}\biggl\{\biggl| \frac{1}{n} (W^\top W)_{1,1} - h_n\biggr| \!>\! (8+\smallO(1)) \sqrt\frac{\log (1/\delta_n)}{n} \bigl( (\kappa_1 + \kappa_2) \sqrt{n/p} + \kappa_1 \bigr)\biggr\}\!\leq\! \delta_n.
    \]
\end{prop}
\begin{proof}
Let $X = QT$ be the QR decomposition of $X$, which is almost surely unique if we require the upper-triangular matrix $T$ to have non-negative entries on the diagonal.

Let $Q_1$ be the submatrix obtained from the first $n_1$ rows of $Q$. From Lemma~\ref{Lemma:GaussianQR}, $Q_1$ and $T$ are independent and $T$ has independent entries distributed as $T_{j,j} = t_j > 0$ with $t_j^2\sim \chi^2_{n-j+1}$ for $j\in[p]$ and $T_{j,k} = z_{j,k} \sim N(0,1)$ for $1\leq j < k\leq p$.

Define $B := Q_1^\top Q_1$ and let $B = V\Lambda V^\top$ be its eigendecomposition, which is almost surely unique if we require the diagonal entries of $\Lambda$ to be non-increasing and the diagonal entries of $V$ to be nonnegative. By Lemma~\ref{Lemma:GaussianQR}, $Q$ is uniformly distributed on $\mathbb{O}^{n\times p}$, which means $Q \stackrel{\mathrm{d}}{=} QH$ for any $H\in\mathbb{O}^{p\times p}$.
Consequently, $Q_1 \stackrel{\mathrm{d}}{=} Q_1 H$ and $B\stackrel{\mathrm{d}}{=} H^\top B H = (H^\top V) \Lambda (H^\top V)^\top$. Since the group $\mathbb{O}^{p\times p}$ acts transitively on itself through left multiplication, the joint density of $V$ and $\Lambda$ must be a function of $\Lambda$ only. In particular, $V$ and $\Lambda$ are independent.

Note that $X_1 = Q_1 T$. Thus, $X_1^\top X_1 = T^\top B T$ and $X_2^\top X_2 = T^\top (I_p-B)T$. By Lemma~\ref{Lemma:GramW}, we have
\begin{align}
W^\top W &= 4 X_1^\top A_1 A_1^\top X_1 = 4 (X_1^\top X_1)(X_1^\top X_1 + X_2^\top X_2)^{-1}(X_2^\top X_2)\nonumber\\
& = 4 T^\top B(I_p - B) T = 4 T^\top V \Lambda(I_p-\Lambda) V^\top T.\label{Eq:WtW}
\end{align}
Let $1\geq \lambda_1\geq \cdots\geq \lambda_p\geq 0$ be the diagonal entries of $\Lambda$. Define $a_j = \lambda_j(1-\lambda_j)$ for $j\in[p]$ and set $a := (a_1,\ldots,a_p)$. We can write $t_1^2 = s_1^2 + r_1^2$ with $s_1^2\sim \chi^2_p$ and $r_1^2\sim\chi^2_{n-p}$ such that $s_1\geq 0$, $r_1\geq 0$ are independent of each other and independent of everything else. By Lemma~\ref{Lemma:GaussianQR}, we have that $G_{j,1}:=s_1 V_{j,1}$ for $j\in[p]$ are independent $N(0,1)$ random variables. Note that
\[
\frac{1}{4}(W^\top W)_{1,1} = \sum_{j=1}^p t_1^2 a_j V_{j,1}^2 = \frac{t_{1}^2}{s_{1}^2} \sum_{j=1}^p a_j G_{j,1}^2.
\]
Let $\delta = \delta_n>0$ be chosen later. By \citet[Lemma~1]{LaurentMassart2000}, applied conditionally on $a$, we have with probability at least $1-6\delta$ that all the following inequalities hold:
\begin{gather*}
\|a\|_1 - 2\|a\|_2\sqrt{\log\frac{1}{\delta}}\leq \sum_{j=1}^p a_j G_{j,1}^2 \leq \|a\|_1 + 2\|a\|_2\sqrt{\log\frac{1}{\delta}} + 2\|a\|_\infty \log\frac{1}{\delta},\\
p-2\sqrt{p \log\frac{1}{\delta}}\leq s_1^2 \leq p+2\sqrt{p\log\frac{1}{\delta}}+2\log\frac{1}{\delta},\\
n-2\sqrt{n\log\frac{1}{\delta}}\leq t_1^2 \leq n+2\sqrt{n\log\frac{1}{\delta}}+2\log\frac{1}{\delta}.
\end{gather*}
Keeping in mind that $\|a\|_\infty\leq 1/4$, we have with probability at least $1-6\delta$ that
\begin{align*}
&\frac{n-2\sqrt{n\log(1/\delta)}}{p+2\sqrt{p\log(1/\delta)}+2\log(1/\delta)}\bigl\{\|a\|_1 - 2\|a\|_2\sqrt{\log(1/\delta)}\bigr\}  \leq
\frac{1}{4}(W^\top W)_{1,1} \\
&\quad \leq \frac{n+2\sqrt{n\log (1/\delta)}+2\log(1/\delta)}{p-2\sqrt{p\log (1/\delta)}}\bigl\{\|a\|_1 + 2\|a\|_2\sqrt{\log(1/\delta)} + \frac{1}{2}\log(1/\delta)\bigr\}
\end{align*}
If $\log(1/\delta) = \smallO(p)$, then for each $p$ with probability at least $1-6\delta$, we have that
\begin{equation}
    \begin{aligned}
      \biggl| \frac{(W^\top W)_{1,1}}{4} & - \frac{n}{p} \| a \|_1 \biggr|   \leq \| a \|_1 \frac{2 \sqrt{n \log(1 / \delta)}}{ p} \bigl( 1 + \!\sqrt{n/p}\bigr) \\
     & + \frac{n}{p} \biggl\{2 \| a \|_2 \sqrt{\log(1/ \delta)} + \frac{\log(1/\delta)}{2} \biggr\} \\
     & + \mathcal{O}_s\biggl(\frac{\|a\|_1\log(1/\delta)}{p}+\frac{\|a\|_2\log(1/\delta)}{\sqrt{p}}+\frac{\log^{3/2} (1/\delta)}{p^{1/2}}\biggr).
    \end{aligned}
     \label{Eq:ErrorTerms}
\end{equation}
By the definition of $B$, we have for $H:=T(X^\top X)^{-1/2}\in\mathbb{O}^{p\times p}$ that
\[
H^\top B H = H^\top T^{-\top}X_1^\top X_1 T^{-1}H = (X^\top X)^{-1/2} (X_1^\top X_1)(X^\top X)^{-1/2},
\]
which follows the matrix-variate Beta distribution $\mathrm{Beta}_p(n_1/2, n_2/2)$ as defined before Lemma~\ref{lemma:spectral-limit}. Hence, the diagonal elements of $\Lambda$ are the same as the eigenvalues of a $\mathrm{Beta}_p(n_1/2, n_2/2)$ random matrix. By~\eqref{Eq:ErrorTerms}, for each $p$, with probability at least $1-6\delta$, we have
\begin{equation}
    \label{Eq:OneVectorTail}
    \begin{aligned}
        \biggl|  (W^\top W)_{1,1}  -\frac{4n}{p}\|a\|_1\biggr| & \leq 8 \sqrt{n \log (1/\delta)} \bigl( (\kappa_1 + \kappa_2) \sqrt{n/p} + \kappa_1 \bigr) \\
        & \qquad \qquad \quad + \mathcal{O}_s\biggl(\log(1/\delta) + \frac{\log^{3/2}(1/\delta)}{p^{1/2}}\biggr).
    \end{aligned}
\end{equation}
The desired result follows by taking $h_n = 4\|a\|_1/p$ and observing that by Lemma~\ref{lemma:spectral-limit} $h_n \to 4\kappa_1$ almost surely.
\end{proof}

With Proposition~\ref{Prop:WtW}, we are now in a position to prove Proposition~\ref{Prop:WtW-general} in its general form.

\begin{proof}[Proof of Proposition~\ref{Prop:WtW-general}]
Let $V_i := X_i \Sigma^{-1/2}, i = 1, 2$ and set $V = (V_1^\top, V_2^\top)^\top$. Then each row of $V$ follows $N(0, I_p)$ and is independent of each other. We have
\begin{align*}
    W^\top W & = 4 (X_1^\top X_1) (X^\top X)^{-1} (X_2^\top X_2) \\
    & = 4 (\Sigma^{1/2} V_1^\top V_1 \Sigma^{1/2}) (\Sigma^{-1/2} (V^\top V)^{-1} \Sigma^{-1/2}) (\Sigma^{1/2} V_2^\top V_2 \Sigma^{1/2}) \\
    & =  4\Sigma^{1/2} V_1^\top V_1 (V^\top V)^{-1} V_2^\top V_2 \Sigma^{1/2} \stackrel{\mathrm{d}}{=} \Sigma^{1/2} (W_I^\top W_I) \Sigma^{1/2},
\end{align*}
where $W_I$ is the complementarily sketched design matrix when all entries of $X$ are independent standard normals (i.e.\ $\Sigma = I_p$). Let $h_n$ be the random sequence satisfying Proposition~\ref{Prop:WtW} and define $\Delta := W^\top W/ n - h_n  \Sigma$ and $\Delta_I := W_I^\top W_I / n - h_n I_p$. We start by controlling the $\ell$-sparse operator norm of $\Delta$ for an arbitrary $\ell\in[p]$. By Lemma~\ref{Lemma:Net}, there exists a $1/4$-net $\mathcal{N}_\ell$ of $\{v\in\mathcal{S}^{p-1}: \|v\|_0\leq \ell\}$ of cardinality at most $\binom{p}{\ell}9^\ell$, such that
\begin{align}
\label{Eq:Prop2tmp1}
\|\Delta\|_{\ell,\mathrm{op}} &= 2\sup_{u\in \mathcal{N}_\ell} u^\top \Delta u \stackrel{\mathrm{d}}{=} 2\sup_{u\in \mathcal{N}_\ell} (\Sigma^{1/2}u)^\top \Delta_I (\Sigma^{1/2}u) \nonumber\\
&\leq 2\|\Sigma \|_{k,\mathrm{op}} \sup_{u\in\mathcal{N}_\ell'} u^\top \Delta_I u \leq 2\overline{\lambda}\sup_{u\in\mathcal{N}_\ell'} u^\top \Delta_I u,
\end{align}
where $\mathcal{N}_\ell' := \{\Sigma^{1/2}u / \|\Sigma^{1/2}u\|_2: u\in\mathcal{N}_\ell\}$. We claim that for any $u\in\mathcal{S}^{p-1}$, we have $u^\top \Delta_I u \stackrel{\mathrm{d}}= e_1^\top \Delta_I e_1$. This is because for any $H\in\mathbb{O}^{p\times p}$, we have $V\stackrel{\mathrm{d}}= VH$ and hence by Lemma~\ref{Lemma:GramW} that
\begin{align}
H^\top W_I^\top W_I H &= 4(H^\top V_1^\top V_1H)(H^\top V^\top VH)^{-1}(H^\top V_2^\top V_2H)\nonumber\\ &\stackrel{\mathrm{d}}
=4(V_1^\top V_1)(V^\top V)^{-1}(V_2^\top V_2)= W_I^\top W_I,\label{Eq:RotationalSymmetry}
\end{align}
which in particular implies our claim.  Consequently, by Proposition~\ref{Prop:WtW} and a union bound, when $\log(1/\delta) = \smallO(p)$, we have with probability at least $1- 6|\mathcal{N}_\ell|\delta$ that
\begin{equation}
\label{Eq:Prop2tmp3}
    \| \Delta\|_{\ell,\mathrm{op}} \leq (16+\smallO(1))\overline{\lambda}\sqrt\frac{\log(1/\delta)}{n}\bigl\{(\kappa_1+\kappa_2)\sqrt{n/p}+\kappa_1\bigr\}.
\end{equation}

To prove~\eqref{Eq:WtWdiag-1}, we recall $\diag(\Sigma) = I_p$ and set $\ell=1$ and $\delta=p^{-3}$ in~\eqref{Eq:Prop2tmp3} to obtain with probability at least $1-54p^{-2}$ that
\begin{equation}
    \begin{aligned}
        \max_{j\in[p]}\biggl|\frac{(W^\top W)_{j,j}}{nh_n} & - 1\biggr|  = \frac{1}{h_n}\|\Delta\|_{1,\mathrm{op}} \\
        & \leq \frac{(16+\smallO(1))\overline{\lambda}}{h_n}\sqrt\frac{3\log p}{n}\bigl\{(\kappa_1+\kappa_2)\sqrt{n/p}+\kappa_1\bigr\}.
    \end{aligned}
\label{Eq:Prop2tmp4}
\end{equation}
The first conclusion follows by noting that $h_n\to 4\kappa_1$ and an application of the Borel--Cantelli lemma (since $p^{-2}$ is summable).

To prove~\eqref{Eq:kopNorm-1}, we set $\ell=k$ and $\delta = (10ep/k)^{-(k+4)}$. By~\eqref{Eq:klogpovern-1}, we have $\log(1/\delta) = (k+4)\log(10ep/k) = \smallO(p)$ and
\begin{align*}
|\mathcal{N}|\delta & \leq 9^k\binom{p}{k}\biggl(\frac{10ep}{k}\biggr)^{-(k+4)}  \leq \biggl(\frac{9ep}{k}\biggr)^k\biggl(\frac{10ep}{k}\biggr)^{-(k+4)} \\
& \leq \frac{0.9^k}{(ep/k)^4} \leq \max(p^{-2}, 0.9^{\sqrt{p}}).
\end{align*}
By the Borel--Cantelli lemma,
\begin{equation}
\label{Eq:Prop2tmp2}
    \| \Delta\|_{k,\mathrm{op}} \leq (16+\smallO(1))\overline{\lambda}\sqrt\frac{(k+4)\log(10ep/k)}{n}\bigl\{(\kappa_1+\kappa_2)\sqrt{n/p}+\kappa_1\bigr\}.
\end{equation}
holds for all but finitely many $p$.  We work on $p$ sufficiently large such that \eqref{Eq:Prop2tmp2} holds henceforth. Define $\hat D:=\diag(W^\top W) / (n h_n)$. By \eqref{Eq:Prop2tmp4} and a Taylor expansion, we have that $\|\hat D^{-1/2} - I\|_{\mathrm{op}} = (1 +\smallO(1))(2h_n)^{-1}\|\Delta\|_{1,\mathrm{op}}$. Thus,
\begin{align*}
\|\hat D^{-1/2} W^\top & W \hat D^{-1/2} -  W^\top W  \|_{k,\mathrm{op}} \\
& \leq \|\hat D^{-1/2} - I\|_{\mathrm{op}}\| W^\top W \|_{k,\mathrm{op}} (1+\|\hat D^{-1/2}\|_{\mathrm{op}})\\
&\leq (2+\smallO(1))\frac{\|\Delta\|_{1,\mathrm{op}}}{2h_n} \|n\Delta+nh_n\Sigma\|_{k,\mathrm{op}}
\leq (1+\smallO(1)){n\overline{\lambda}\|\Delta\|_{1,\mathrm{op}}}%{2}.
\end{align*}
where the final bound follows from~\eqref{Eq:Prop2tmp2} and the fact that $\|\Sigma\|_{k,\mathrm{op}}\leq \overline{\lambda}$. Consequently, noting $h_n\to 4\kappa_1$, for all large $p$ we have
\begin{align*}
\|\tilde W^\top \tilde W - \Sigma\|_{k,\mathrm{op}} &= \frac{1}{h_n} \|\hat D^{-1/2} W^\top W \hat D^{-1/2}/n - h_n \Sigma\|_{k,\mathrm{op}} \\
&\leq \frac{1}{h_n}\|\Delta\|_{k,\mathrm{op}} + \frac{1}{nh_n}\|\hat D^{-1/2} W^\top W \hat D^{-1/2} -  W^\top W  \|_{k,\mathrm{op}} \\
&\leq \frac{1}{h_n} (\|\Delta\|_{k,\mathrm{op}} + \overline{\lambda} \|\Delta\|_{1,\mathrm{op}}) \\
& \leq C_{s,r}\biggl\{\overline{\lambda}\sqrt\frac{k\log(ep/k)}{n}+\overline{\lambda}^2\sqrt\frac{\log p}{n}\biggr\},
\end{align*}
for $C_{s,r} := 9 (1+\sqrt{1+1/s}+\sqrt{s+r-1+1/s+1/r})$, which completes the proof.
\end{proof}

\begin{proof}[Proof of Theorem~\ref{Thm:Alternative}]
By Proposition~\ref{Prop:WtW-general}, it suffices to work with a deterministic sequence of $W$ such that~\eqref{Eq:WtWdiag-1} and \eqref{Eq:kopNorm-1} holds, which we henceforth assume in this proof.

Define $\tilde\theta = (\tilde\theta_1,\ldots,\tilde\theta_p)^\top$ such that $\tilde\theta_j := \theta_j \|W_j\|_2$. Then, from~\eqref{Eq:ZW}, we have
\[
Z = \tilde W \tilde\theta + \xi,
\]
for $\xi \sim N_m(0, I_m)$. Write $Q := (Q_1,\ldots,Q_p)^\top$ and $S := \supp(\theta)=\supp(\tilde\theta)$, then
\[
Q_S = (\tilde W^\top Z)_S \sim N_k\bigl((\tilde W^\top \tilde W)_{S,S}\tilde\theta_S,\, (\tilde W^\top \tilde W)_{S,S}\bigr).
\]
Our strategy will be to control $\|Q_S\|_2^2$. To this end, we first look at the quantity $\|(\tilde W^\top \tilde W)_{S,S}^{-1/2}Q_S\|_2^2$, which has a noncentral chi-squared distribution $\chi^2_k(\|\tilde\theta_S^\top (\tilde W^\top \tilde W)_{S,S}\tilde\theta_S\|_2^2)$. By~\eqref{Eq:WtWdiag-1} and~\eqref{Eq:kopNorm-1}, we have
\begin{align*}
\|\tilde\theta_S^\top (\tilde W^\top \tilde W)_{S,S}\tilde\theta_S\|_2^2 \geq \|(\tilde W^\top \tilde W)_{S,S}\|_{\mathrm{op}} \|\theta_S\|_2^2 &\geq \{\underline{\lambda}-\smallO(1)\}4n\kappa_1\rho^2 \\
&\geq \frac{28\sigma^2 k\log p}{\underline{\lambda}},
\end{align*}
where we have used the fact that $\rho^2 \geq 7k\log p /(\underline{\lambda}^2 n\kappa_1)$ in the final bound. Thus, by \citet[Lemma~8.1]{Birge2001}, we have with probability at least $1-p^{-2}$ that
\begin{align}
\|(\tilde W^\top \tilde W)_{S,S}^{-1/2}Q_S\|_2^2& \nonumber\\
&\hspace{-1.8cm}\geq k + \|\tilde\theta_S^\top (\tilde W^\top \tilde W)_{S,S}\tilde\theta_S\|_2^2 - 2\sqrt{(2k+4\|\tilde\theta_S^\top (\tilde W^\top \tilde W)_{S,S}\tilde\theta_S\|_2^2)\log p}\nonumber\\
&\hspace{-1.8cm}\geq \{1-\smallO(1)\}\|\tilde\theta_S^\top (\tilde W^\top \tilde W)_{S,S}\tilde\theta_S\|_2^2 - 4\|\tilde\theta_S^\top (\tilde W^\top \tilde W)_{S,S}\tilde\theta_S\|_2\sqrt{\log p}\nonumber\\
&\hspace{-1.8cm}\geq \frac{(6-\smallO(1)) \sigma^2 k\log p}{\underline{\lambda}}. \label{Eq:tmp2}
\end{align}
Consequently, by~\eqref{Eq:tmp2} and~\eqref{Eq:kopNorm-1} again, we have with probability at least $1-p^{-2}$ that
\begin{align}
\label{Eq:tmp4}
\|Q_S\|_2^2&\geq \|(\tilde W^\top \tilde W)_{S,S}\|_{\mathrm{op}}\|(\tilde W^\top \tilde W)_{S,S}^{-1/2}Q_S\|_2^2 \nonumber\\
&\geq \{\underline{\lambda}-\smallO(1)\}\frac{6 \sigma^2 k\log p}{\underline{\lambda}} \geq (6-\smallO(1))\sigma^2 k\log p.
\end{align}
By Condition~\ref{cond:sigma}, we define the almost sure event $ \Omega_\sigma := \{  \hat{\sigma}/\sigma = 1 + \smallO(1) \}$ and $\omega_0 :=  \sigma \sqrt{(4+\varepsilon)\log p}$.  We observe that on the event $\Omega_\sigma$,
\(
    ({\omega}/{\omega_0})^2 - 1 = (\hat{\sigma} /\sigma +1) (\hat{\sigma}/ \sigma -1)  = \smallO(1).
\)
From~\eqref{Eq:tmp4},  using the tuning parameters $\omega=\hat{\sigma}\sqrt{(4+\varepsilon)\log p}$ and $\tau \leq  \hat{\sigma}^{2} k\log p$,  we have, conditionally on $\Omega_\sigma$, for sufficiently large $p$ that with probability at least $1-2p^{-2}$,
\[
T = \sum_{j=1}^pQ_j^2\mathbbm{1}_{\{|Q_j|\geq \omega\}} \geq \|Q_S\|_2^2 - k\omega_0^2 \bigl( 1 + (\omega/\omega_0)^2 - 1 \bigr) \geq \hat\sigma^{2}  k\log p \geq \tau.
\]
which would allow us to reject the null on the `good' event $\Omega_\sigma$.
Keeping in mind that $\Omega_\sigma$ is an almost sure event, we proceed to bound
\begin{align*}
    \mathbb{P}( T \le \tau)  = \mathbb{P} ( T \le \tau \mid \Omega_\sigma)  \le 2p^{-2},
\end{align*}
The desired almost sure convergence follows by the Borel--Cantelli lemma since $1/p^2$ is summable over $p \in\mathbb{N}$.
\end{proof}

\begin{proof}[Proof of Corollary~\ref{Cor:SparseUpper}]
The first inequality follows from the definition of~$\mathcal{M}_X(k,\rho)$. An inspection of the proofs of Theorems~\ref{Thm:Null} and~\ref{Thm:Alternative} of the main text reveals that both results only depend on the complementary-sketched model $Z=W\theta+\xi$, and hence hold uniformly over $(\beta_1,\beta_2)$. Thus, we have from Theorem~\ref{Thm:Null} that $\sup_{\beta\in\mathbb{R}^p} P_{\beta,\beta}^X (\psi^{\mathrm{sparse}}_{\lambda,\tau}\neq 0) \xrightarrow{\mathrm{a.s.}} 0$ and from Theorem~\ref{Thm:Alternative} that $\sup_{\beta_1,\beta_2\in\mathbb{R}^p: (\beta_1-\beta_2)/2\in\Theta_{p,k}(\rho)} P_{\beta,\beta}^X (\psi^{\mathrm{sparse}}_{\lambda,\tau}\neq 1) \xrightarrow{\mathrm{a.s.}} 0$. Combining the two completes the proof.
\end{proof}

\begin{proof}[Proof of Theorem~\ref{Thm:LowerBound}]
By considering a trivial test $\tilde\psi \equiv 0$, we see that $\mathcal{M}\leq 1$. Thus, it suffices to show that $\mathcal{M}\geq 1-\smallO(1)$. Note that since $k\leq p^{1/2}$, condition~\eqref{Eq:klogpovern-1} is satisfied and hence by Proposition~\ref{Prop:WtW-general}, it suffices to work with a deterministic sequence of $X$ (and hence $W$) such that~\eqref{Eq:WtWdiag-1} and \eqref{Eq:kopNorm-1} holds, which we henceforth assume in this proof.

It is convenient to reparametrize the distributions in terms of $(\gamma,\theta) = ((\beta_1+\beta_2)/2, (\beta_1-\beta_2)/2)$ instead of $(\beta_1,\beta_2)$. Define $Q^X_{\gamma,\theta}:=P^X_{\beta_1, \beta_2}$. Let $L := (X_1^\top X_1 + X_2^\top X_2)^{-1} (X_2^\top X_2 - X_1^\top X_1) $ and $\pi$ be the uniform distribution on
\[
\Theta_0 :=\{\theta\in\{k^{-1/2}\rho, -k^{-1/2}\rho, 0\}^p:\|\theta\|_0=k\} \subseteq \Theta.
\]
We write $Q_0:=Q^X_{0,0}$ and let $Q_\pi:= \int_{\theta\in\Theta_0}Q^X_{L\theta, \theta}\,d\pi(\theta)$ denote the uniform mixture of $Q^X_{\gamma,\theta}$ for $\{(\gamma,\theta):\theta\in\Theta_0,\gamma=L\theta\}$. Let $\mathcal{L}:=dQ_{\pi}/dQ_0$ be the likelihood ratio between the mixture alternative $Q_\pi$ and the simple null $Q_0$. We have that
\begin{align*}
\mathcal{M} &\geq \inf_{\tilde\psi} \Bigl\{1 - (Q_0 - Q_{\pi}) \tilde\psi \Bigr\} = 1 - \frac{1}{2}\int \biggl|1 - \frac{dQ_{\pi}}{dQ_0}\biggr| dQ_0\\
& \geq 1- \frac{1}{2}\biggl\{\int \biggl(1 - \frac{dQ_{\pi}}{dQ_0}\biggr)^2 dQ_0\biggr\}^{1/2} \geq 1 - \frac{1}{2} \{Q_0(\mathcal{L}^2)-1\}^{1/2}.
\end{align*}
So it suffices to prove that $Q_0(\mathcal{L}^2) \leq 1+\smallO(1)$. Writing $\tilde X_1 = X_1L+X_1$ and $\tilde X_2 = X_2L-X_2$, by the definition of $Q_\pi$, we compute that
\begin{align*}
\mathcal{L} = \int \frac{dQ^X_{L\theta,\theta}}{dQ_{0}}\,d\pi(\theta) &= \int \frac{e^{-\frac{1}{2}(\|Y_1-X_1L\theta-X_1\theta\|^2+\|Y_2 - X_2L\theta+X_2\theta\|^2)}}
{e^{-\frac{1}{2}(\|Y_1\|^2+\|Y_2\|^2)}}\,d\pi(\theta)\\
&=\int e^{\langle \tilde X_1\theta, Y_1\rangle - \frac{1}{2}\|\tilde X_1\theta\|^2 +\langle \tilde X_2\theta, Y_2\rangle - \frac{1}{2}\|\tilde X_2\theta\|^2}\,d\pi(\theta).
\end{align*}
For $\theta\sim \pi$ and some fixed $J_0\subseteq [p]$ with $|J_0|=k$, let  $\pi_{J_0}$ be the distribution of $\theta_{J_0}$ conditional on $\supp(\theta) = J_0$. Let $J,J'$ be independently and uniformly distributed on $\{J_0\subseteq [p]: |J_0|=k\}$. Define $\tilde\theta := (\tilde\theta_1,\ldots,\tilde\theta_p)^\top$ and $\tilde\theta':=(\tilde\theta'_1,\ldots,\tilde\theta'_p)^\top$ such that $\tilde\theta_j := \theta_j \|W_j\|_2$ and $\tilde\theta'_j:=\theta'_j\|W_j\|_2$. Since $\Sigma\in\mathcal{C}(D)$, we can write $\Sigma = \Sigma_0+\Gamma$ for $\Sigma_0\in\mathrm{RowSp}(D)$ and $\|\Gamma\|_{\max} = \smallO(D / (k\log p))$. Also, since $\Sigma$ is symmetric and $\|\Sigma\|_{\max} = \|\diag(\Sigma)\|_{\max} = 1$, we may assume without loss of generality that $\Sigma_0$ is symmetric and $\|\Sigma_0\|_{\max} \leq 1$. By Fubini's theorem and Lemmas~\ref{Lemma:MessingAround} and~\ref{Lemma:GramW}, we have
\begin{align}
Q_{0} (\mathcal{L}^2) &= \iint_{\theta,\theta'} e^{\frac{1}{2} \|\tilde X_1(\theta+\theta')\|^2 - \frac{1}{2}\|\tilde X_1\theta\|^2- \frac{1}{2}\|\tilde X_1\theta'\|}\nonumber \\
&\hspace{3cm}\times e^{ \frac{1}{2} \|\tilde X_2(\theta+\theta')\|^2 - \frac{1}{2}\|\tilde X_2\theta\|^2- \frac{1}{2}\|\tilde X_2\theta'\|^2}\,d\pi(\theta) \,d\pi(\theta')\nonumber\\
& = \iint_{\theta,\theta'} e^{\theta^\top (\tilde X_1^\top \tilde X_1 + \tilde X_2^\top \tilde X_2)  \theta'}\,d\pi(\theta) \,d\pi(\theta') \nonumber\\
&= \iint_{\theta,\theta'} e^{\theta^\top W^\top W  \theta'}\,d\pi(\theta) \,d\pi(\theta')\nonumber \\
&\leq \{\mathbb{E}(e^{2\tilde\theta^\top (\tilde W^\top \tilde W  - \Sigma_0) \tilde \theta'})\}^{1/2}\{\mathbb{E}(e^{2\tilde\theta^\top \Sigma_0 \tilde \theta'})\}^{1/2} ,\label{Eq:LambdaSqr}
\end{align}
where we apply the Cauchy--Schwarz inequality in the final inequality. We now bound the two factors in the final expression separately. By~\eqref{Eq:WtWdiag-1}, we have
\begin{equation}
\label{Eq:rhotovartheta}
\vartheta:=\max\bigl\{\max_{j\in J} |\tilde\theta_j|, \max_{j\in J'} |\tilde\theta'_j|\bigr\} \leq (1+\smallO(1))\sqrt\frac{4n\kappa_1\rho^2}{k}.
\end{equation}
By~\eqref{Eq:kopNorm-1}, we have that
\begin{align}
\label{Eq:FactorI_tbc}
\vartheta^2  \|(\tilde W^\top \tilde W &- \Sigma_0)_{J,J'}\|_{\mathrm{F}} \leq \sqrt{k}\vartheta^2\|(\tilde W^\top \tilde W - \Sigma)_{J,J'}\|_{\mathrm{op}} + \vartheta^2\|\Gamma_{J,J'}\|_{\mathrm{F}}\nonumber\\
& \hspace{-1cm} \leq \vartheta^2\bigl\{\sqrt{k}\|\tilde W^\top \tilde W - \Sigma\|_{2k,\mathrm{op}} + k\|\Gamma\|_{\max}\bigr\}\nonumber\\
& \hspace{-1cm}\leq  \frac{(4+\smallO(1))n\kappa_1\rho^2}{k} \biggl[C_{s,r}\biggl\{\overline{\lambda}\sqrt\frac{k^2\log(ep)}{n}+\overline{\lambda}^2\sqrt\frac{k\log p}{n}\biggr\} + \smallO\biggl(\frac{D}{\log p}\biggr)\biggr].
\end{align}
Since $\rho\leq \sqrt\frac{(1-2\alpha-\varepsilon)k\log p}{4Dn\kappa_1}$, we have from~\eqref{Eq:FactorI_tbc} that $\vartheta^2 \|(\tilde W^\top \tilde W - \Sigma_0)_{J,J'}\|_{\mathrm{F}}=\smallO(1)$. Consequently, by Lemma~\ref{Lemma:Chaos}, we have for sufficiently large $p$ that
\begin{align}
\mathbb{E}(e^{2\tilde\theta^\top (\tilde W^\top \tilde W  - \Sigma_0) \tilde \theta'}) &\leq 1 + C\vartheta^2\|(\tilde W^\top \tilde W - \Sigma_0)_{J,J'}\|_{\mathrm{F}} e^{4\vartheta^2\|(\tilde W^\top \tilde W - \Sigma_0)_{J,J'}\|_{\mathrm{F}}^2} \nonumber\\ &= 1+\smallO(1).
\label{Eq:FactorI}
\end{align}
For the second factor on the right-hand side of~\eqref{Eq:LambdaSqr}, we have by Lemma~\ref{Lemma:Chaos2} that
\begin{equation}
\mathbb{E}(e^{2\tilde\theta^\top \Sigma_0\tilde\theta'}) \leq  \biggl\{1 + \frac{Dk}{p}\bigl(\cosh(2\vartheta^2 D) - 1\bigr)\biggr\}^k \leq \exp\biggl(\frac{Dk^2}{p}e^{2\vartheta^2 D} \biggr). \label{Eq:FactorII}
\end{equation}
Since $\alpha \in [0,1/2)$ and $\rho^2 \leq \frac{(1-2\alpha-\varepsilon)k\log p}{8Dn\kappa_1}$, from~\eqref{Eq:rhotovartheta}, we deduce that
\[
2\vartheta^2D \leq (1-2\alpha-\varepsilon+\smallO(1))\log p
\]
and hence $e^{2\vartheta^2 D}D k^2/p = p^{-\varepsilon+\smallO(1)} = \smallO(1)$. So, from~\eqref{Eq:LambdaSqr}, \eqref{Eq:FactorI} and~\eqref{Eq:FactorII} we have $Q_0(\mathcal{L}^2) \leq  1+\smallO(1)$, which completes the proof.
\end{proof}

\begin{proof}[Proof of Theorem~\ref{Thm:AlternativeDense}]
%\fn{Here we prove a slight stronger version of Theorem~\ref{Thm:AlternativeDense}, where $\eta :=\hat\sigma^2 ( m + 2 \sqrt{m(1+\varepsilon)\log p} + 2 (1 + \varepsilon) \log p) $ for any $0 < \varepsilon < 5$.}
As in the proof of Theorem~\ref{Thm:Alternative}, we work with a deterministic sequence of $W$ such that~\eqref{Eq:WtWdiag-1} and \eqref{Eq:kopNorm-1} are satisfied.
Furthermore, by Condition~\ref{cond:sigma-strong}, we henceforth work on the almost sure event $\Omega_\sigma = \{ |\hat\sigma/\sigma - 1 | = \smallO(p^{-1/2}\log^{1/2} p) \}$.
For $\tilde\theta = (\tilde\theta_1,\ldots,\tilde\theta_p)^\top$ such that $\tilde\theta_j := \theta_j \|W_j\|_2$, we have from~\eqref{Eq:ZW} that
\(
Z = \tilde W \tilde\theta + \xi,
\)
for $\xi \sim N_m(0, \sigma^2 I_m)$. Hence, under the null hypothesis,  $(\|Z\|_2^2/\sigma^2) \sim \chi^2_m$, which by \citet[Lemma~1]{LaurentMassart2000} yields that
\[
    \mathbb{P}\bigl\{\|Z\|_2^2 / \sigma^2 \geq m + 2\sqrt{m\log(1/\delta)} + 2\log(1/\delta) \bigr\} \leq \delta.
\]
%Setting $\delta=p^{-2}$, for $\eta=\hat{\sigma} (m +2^{3/2}\sqrt{m\log p}+4\log p)$, we have by the Borel--Cantelli lemma that $\psi^{\mathrm{dense}}_\eta(X_1,X_2,Y_1,Y_2) \xrightarrow{\mathrm{a.s.}} 0$.
We set $\delta = p^{- (1+\varepsilon/2)}$ and have for $\eta_0=\sigma^2 (m +2\sqrt{(1+\varepsilon/2)m\log p}+2(1+\varepsilon/2)\log p)$, $\eta = \hat\sigma^2 ( m + 2 \sqrt{m(1+\varepsilon)\log p} + 2 (1 + \varepsilon) \log p)  \ge \eta_0$ on $\Omega_\sigma$ for all $p$ sufficiently large.
We bound, for all $p$ sufficiently large,
\begin{equation*}
    \mathbb{P} ( \| Z \|^2 \ge \eta) \le \mathbb{P}( \|Z\|^2 \ge \eta_0) \le p^{-(1+\varepsilon/2)},
\end{equation*}
whence, by the Borel--Cantelli lemma, we have $\psi^{\mathrm{dense}}_\eta(X_1,X_2,Y_1,Y_2) \xrightarrow{\mathrm{a.s.}} 0$.

On the other hand, under the alternative, $(\|Z\|_2^2/\sigma^2) \sim \chi^2_m(\|W\theta\|_2^2)$. Observe from~\eqref{Eq:WtWdiag-1} and~\eqref{Eq:kopNorm-1} that
\begin{align*}
        \|W\theta\|_2^2&=\|\tilde W\tilde \theta\|_2^2 = \|\Sigma^{1/2}\tilde\theta\|_2^2 + \tilde\theta^\top (\tilde W^\top \tilde W - \Sigma)\tilde\theta \\
        &\geq  \|\tilde\theta\|_2^2\bigl(\underline{\lambda}-\|\tilde W^\top \tilde W - \Sigma\|_{k,\mathrm{op}}\bigr)\\
&\geq (4\underline{\lambda}-\smallO(1))n\kappa_1\rho^2 \geq (8-\smallO(1))\sigma^2 \sqrt{m\log p}.
\end{align*}
By a similar argument, we also have $\|W\theta\|_2^2 \leq (4\overline{\lambda}+\smallO(1))n\kappa_1\rho^2 = \smallO(m)$.
Thus, by \citet[Lemma~8.1]{Birge2001}, we have with probability at least $1-p^{-2}$ that
\begin{align*}
\|Z\|_2^2/\sigma^2 &\geq  m + \|W\theta\|_2^2 - 2\sqrt{(2m+4\|W\theta\|_2^2)\log p}\\
&\geq m + (8-2\sqrt{2}-\smallO(1))\sqrt{m\log p}
\end{align*}
%which is at least $\eta = \hat\sigma^2 ( m+\sqrt{8 m\log p}+4\log p)$ for all sufficiently large $p$. 
which is at least $\eta = \hat\sigma^2 ( m + 2\sqrt{2(1+\varepsilon)m\log p} + 2(1+\varepsilon) \log p) \le $ for all $p$ sufficiently large and any $\varepsilon\in (0, 5]$, conditionally on the almost sure event~$\Omega_\sigma$. 

Consequently, we have $\mathbb{P}(\|Z\|_2^2\leq \eta)\leq p^{-2}$ for all large $p$. As $p^{-2}$ is summable, by the Borel--Cantelli lemma,  $\psi^{\mathrm{dense}}_\eta(X_1,X_2,Y_1,Y_2)\xrightarrow{\mathrm{a.s.}}1$.
%\fn{Take $\varepsilon=1$ in the preceding arguments gives the theorem statement where $\eta = \hat\sigma^2( m + 2\sqrt{2m\log p} + 4 \log p)$.}
\end{proof}

\begin{proof}[Proof of Theorem~\ref{Thm:LowerBoundDense}]
Note that since $k\leq p^{\alpha}$ for $\alpha<1$, condition~\eqref{Eq:klogpovern-1} is satisfied and hence we can work with a deterministic sequence of $W$ satisfying~\eqref{Eq:WtWdiag-1} and~\eqref{Eq:kopNorm-1}. Similar to the proof of Theorem~\ref{Thm:LowerBound}, we write $\Sigma = \Sigma_0 + \Gamma$ for some $\Sigma_0\in\mathrm{RowSp}(D)$ with $\|\Sigma_0\|_{\max}\leq 1$ and $\|\Gamma\|_{\max}=\smallO(D/(k\log p))$. We follow the proof of Theorem~\ref{Thm:LowerBound} up to~\eqref{Eq:FactorI_tbc}. 

Now, noting the assumption on $\rho^2$, we have by~\eqref{Eq:rhotovartheta} that  $\mathbb{E}(e^{2\tilde\theta^\top (\tilde W^\top \tilde W - \Sigma_0)\tilde\theta'}) = 1+\smallO(1)$. It remains to show that  $\mathbb{E}(e^{2\tilde\theta^\top  \Sigma_0\tilde\theta'}) = 1+\smallO(1)$. To this end, we have by Lemma~\ref{Lemma:Chaos2} that
\begin{align*}
\mathbb{E}(e^{2\tilde\theta^\top \Sigma_0\tilde\theta'}) &\leq  \biggl\{1 + \frac{Dk}{p}\bigl(\cosh(2\vartheta^2 D) - 1\bigr)\biggr\}^k \leq \biggl\{1 + \frac{Dk}{p} \bigl(2+\smallO(1)\bigr)\vartheta^4 D^2\biggr\}^k\\
& \leq \exp\biggl\{\frac{(2+\smallO(1))D^3k^2\vartheta^4}{p} \biggr\} = 1+\smallO(1),
\end{align*}
where the second inequality follows by the Taylor expansion of $x\mapsto \cosh(x)$ and the fact that $\vartheta^2 D = (4+\smallO(1))n\kappa_1\rho^2D/k = \smallO(1)$, and the last equality holds by noting $D^3k^2\vartheta^4/p = (16+\smallO(1)) \kappa_1^2 D^3\rho^4 n^2/p = \smallO(1)$. %This completes the proof.
\end{proof}

%!TEX root = ../compsket_supp.tex

\section{Ancillary results}
\label{sec:ancillary}
\begin{prop}
\label{Prop:MXeq1}
If $k = p \geq \min\{n_1,n_2\}$, then $\mathcal{M}_X(k,\rho) = 1$. If $k=p$ and $p/n_1, p/n_2 \in [\varepsilon,1)$ for any fixed $\varepsilon\in(0,1)$, and $\theta=(\beta_1-\beta_2)/2\in \Theta_{p,k}(\rho)$ with
\[
\rho^2 = \smallO\biggl(\max\biggl\{\frac{p}{(n_1-p)^2}, \frac{p}{(n_2-p)^2}\biggr\}\biggr),
\]
then $\mathcal{M}_X(k,\rho)\xrightarrow{\mathrm{a.s.}} 1$.
\end{prop}
\begin{proof}
As in the proof of Theorem~\ref{Thm:LowerBound}, it suffices to control $P_0(\mathcal{L}^2)$ for some choice of prior $\pi$. We write $\lambda_{\min}(W^\top W)$ for the minimum eigenvalue of $W^\top W$ and let $\theta$ be an associated eigenvector with $\ell_2$ norm equal to $\rho$. We choose $\pi$ to be the Dirac measure on $\theta$. Then by~\eqref{Eq:LambdaSqr}, we have
\[
P_0(\mathcal{L}^2) = e^{\theta^\top W^\top W\theta} = e^{\rho^2\lambda_{\min}(W^\top W)}.
\]

When $p\geq n_1$ or $p\geq n_2$, we have by Lemma~\ref{Lemma:GramW} that the Gram matrix
$W^\top W = (X_1^\top X_1)(X^\top X)^{-1}(X_2^\top X_2)$ is singular. Hence, $\lambda_{\min}(W^\top W) = 0$ and $P_0(\mathcal{L}^2) = 1$, which implies that $\mathcal{M}_X(k,\rho) = 1$.

On the other hand, if $p < \min\{n_1,n_2\}$, let $V_1, V_2, V, W_I$ be defined as in the proof of Proposition~\ref{Prop:WtW-general}. Let $T$ and $\Lambda$ be defined as in the proof of Proposition~\ref{Prop:WtW}, with $V$ and $W_I$ taking the roles of $X$ and $W$ therein respectively. Then, by~\eqref{Eq:WtW}, we have
\begin{equation}
\begin{aligned}
\lambda_{\min}(W_I^\top W_I) & \leq 4\|T\|_{\mathrm{op}}^2\lambda_{\min}\bigl(\Lambda(I-\Lambda)\bigr) \\
& \leq 4\|V^\top V\|_{\mathrm{op}} \min\{\lambda_{\min}(\Lambda), 1-\lambda_{\max}(\Lambda)\}.
\end{aligned}
\label{eqn:lambda-min}
\end{equation}
Applying tail bounds for operator norm of a random Gaussian matrix \citep[see, e.g.][Theorem~6.1]{Wainwright2019}, we have
\[
\|V^\top V\|_{\mathrm{op}} \leq n\biggl(1 + \sqrt\frac{p}{n} + \sqrt\frac{2\log p}{n}\biggr)^2 \leq 5n
\]
asymptotically with probability 1. Moreover, by \citet[Theorem~1.1]{BaiHuPanZhou2015}, there is an almost sure event on which the empirical spectral distribution of $\Lambda$ converges weakly to a distribution supported on $[t_\ell,t_r]$, for $t_\ell$ and $t_r$ defined in~\eqref{Eq:tlr}. We will work on this almost sure event henceforth. For $p/n_1\to\xi\in[\varepsilon,1)$ and $p/n_2\to\eta\in[\varepsilon,1)$, we have $\limsup_{p\to\infty} \lambda_{\min}(\Lambda) \leq t_\ell$ and  $\liminf_{p\to\infty} \lambda_{\max}(\Lambda) \geq t_r$. On the other hand, Taylor expanding the expression for $t_\ell$ and $t_r$ in~\eqref{Eq:tlr} with respect to $1-\xi$ and $1-\eta$ respectively, we obtain that
\begin{align*}
t_\ell &= \frac{1}{4}\eta(1-\xi)^2 + \mathcal{O}_\varepsilon\bigl((1-\xi)^3\bigr),\\
1-t_r &= \frac{1}{4}\xi(1-\eta)^2 + \mathcal{O}_\varepsilon\bigl((1-\eta)^3\bigr).
\end{align*}
Therefore, $\min\{\lambda_{\min}(\Lambda), 1-\lambda_{\max}(\Lambda)\} = \mathcal{O}_\varepsilon(\min\{(1-\xi)^2,(1-\eta)^2\})$. By the condition on $\rho^2$ and \eqref{eqn:lambda-min}, we have
\begin{align*}
\rho^2\lambda_{\min}(W^\top W) &\leq \overline{\lambda}\rho^2 \lambda_{\min}(W_I^\top W_I)
\; = \smallO(1),
%\\
%&= \smallO\biggl(\max\biggl\{\frac{1}{(1-\xi)^2p}, \frac{1}{(1-\eta)^2p}\biggr\}\biggr) \mathcal{O}_\varepsilon(n \min\{(1-\xi)^2, (1-\eta)^2\})
\end{align*}
which implies that $P_0(\mathcal{L}^2) = 1+\smallO(1)$ and $\mathcal{M}_X\xrightarrow{\mathrm{a.s.}} 1$.
\end{proof}

\begin{lemma}
\label{Lemma:GramW}
Let $n_1,n_2,p,m$ be positive integers such that $n_1+n_2=p+m=n$. Let $X = (X_1^\top, X_2^\top)^\top \in\mathbb{R}^{n\times p}$ be a non-singular matrix with block components $X_1\in\mathbb{R}^{n_1\times p}$ and $X_2\in\mathbb{R}^{n_2\times p}$.  Choose $A_1\in\mathbb{R}^{n_1\times m}$ and $A_2\in\mathbb{R}^{n_2\times m}$ to satisfy~\eqref{Eq:ConditionA}. Then
\[
X_1^\top A_1 A_1^\top X_1 = -X_2^\top A_2 A_2^\top X_2 = (X_1^\top X_1)(X^\top X)^{-1}(X_2^\top X_2).
\]
\end{lemma}
\begin{proof}
The first equality follows immediately from~\eqref{Eq:ConditionA}. Define $\tilde X_1 := X_1(X^\top X)^{-1/2}$ and $\tilde X_2 := X_2(X^\top X)^{-1/2}$. Then $\tilde X := (\tilde{X}_1^\top, \tilde{X}_2^\top)^\top$ has orthonormal columns with the same column span as $X$, and so
\[
\begin{pmatrix} \tilde X_1 & A_1\\ \tilde X_2 & A_2\end{pmatrix} \in \mathbb{O}^{n\times n}.
\]
In particular, $\tilde X_1 \tilde X_1^\top + A_1 A_1^\top = I_{n_1}$. Therefore,
\begin{align*}
X_1^\top A_1 A_1^\top X_1 &= X_1^\top (I_{n_1} - \tilde X_1\tilde X_1^\top)X_1 = X_1^\top X_1 - X_1^\top X_1(X^\top X)^{-1} X_1^\top  X_1\\
&= X_1^\top X_1(X^\top X)^{-1}(X^\top X - X_1^\top X_1) \\
& = (X_1^\top X_1)(X^\top X)^{-1}(X_2^\top X_2),
\end{align*}
where the last equality holds by noting the block structure of $X$.
\end{proof}
\begin{lemma}
\label{Lemma:MessingAround}
For $X_1\in\mathbb{R}^{n_1\times p}$ and $X_2\in\mathbb{R}^{n_2\times p}$, define $L:=(X_1^\top X_1 + X_2^\top X_2)^{-1} (X_2^\top X_2 - X_1^\top X_1)$, $\tilde X_1 := X_1 (L + I_p)$ and $\tilde{X}_2 := X_2 (L - I_p)$. We have
\[
\tilde X_1^\top \tilde X_1 + \tilde X_2^\top \tilde X_2 = 4 X_1^\top X_1 ( X_1^\top X_1 + X_2^\top X_2)^{-1} X_2^\top X_2.
\]
\end{lemma}
\begin{proof}
Write $G_1 := X_1^\top X_1$, $G_2 := X_2^\top X_2$. It is clear that
\begin{align*}
    L - I_p &  = - 2 (X_1^\top X_1 + X_2^\top X_2)^{-1} X_1^\top X_1   = - 2 (G_1 + G_2)^{-1} G_1, \\
    L + I_p & = 2 (X_1^\top X_1 + X_2^\top X_2)^{-1} X_2^\top X_2 = 2 (G_1 + G_2)^{-1} G_2.
 \end{align*}
 Therefore, we have
 \begin{align*}
     \frac{1}{4}( & \tilde{X}_1^\top  \tilde{X}_1 + \tilde{X}_2^\top \tilde{X}_2 )\\
     & = \frac{1}{4} \bigl\{ (L + I_p)^\top X_1^\top X_1 (L + I_p) + (L - I_p)^\top X_1^\top X_2 (L - I_p) \bigr\} \\
     & =   G_2 (G_1 + G_2)^{-1} G_1 (G_1 +  G_2)^{-1} G_2 \\
     & \qquad +  G_1 (G_1 + G_2)^{-1} G_2 (G_1 + G_2)^{-1} G_1 \\
%     & =  ( G_1 + G_2 -  G_1)(G_1 + G_2)^{-1} G_1 (G_1 +  G_2)^{-1} G_2 \\
%     & \qquad +  G_1 (G_1 + G_2)^{-1} G_2 (G_1 + G_2)^{-1} (G_1  + G_2 - G_2)\\
     & =   -  G_1(G_1 + G_2)^{-1} G_1 (G_1 +  G_2)^{-1} G_2 \\
     & \qquad - G_1(G_1 + G_2)^{-1} G_2 (G_1 +  G_2)^{-1} G_2+ 2G_1 (G_1 + G_2)^{-1} G_2 \\
     & = G_1 (G_1 + G_2)^{-1} G_2.
 \end{align*}
 The proof is complete by recalling the definitions of $G_1$ and $G_2$.
\end{proof}

The following lemma concerns the control of the $k$-operator norm of a symmetric matrix. Similar results have been derived in previous works \citep[see, e.g.][Lemma~2]{wang2016statistical}. For completeness, we include a statement and proof of the specific version we use.
\begin{lemma}
\label{Lemma:Net}
For any symmetric matrix $M\in\mathbb{R}^{p\times p}$ and $k\in[p]$, there exists a subset $\mathcal{N}\subseteq \mathcal{S}^{p-1}$ such that $|\mathcal{N}|\leq \binom{p}{k}9^k$ and
\[
\|M\|_{k, \mathrm{op}} \leq 2 \sup_{u \in \mathcal{N}} u^\top M u.
\]
\end{lemma}
\begin{proof}
Define $\mathcal{B}_0(k) := \cup_{J \subset [p], |J|=k} S_{J}$, where $S_J := \{ v \in \mathcal{S}^{p-1}: v_i = 0, \forall i \notin J \}$.
For each $S_J$, we find a $1/4$-net $\mathcal{N}_J$ of cardinality at most $9^k$ \citep[Lemma~5.2]{Vershynin2012}.
Define $\mathcal{N} := \cup_{J \subset [p], | J| = k} \mathcal{N}_J$, which has the desired upper bound on cardinality.
By construction, for $v \in \argmax_{u\in\mathcal{B}_0(k)} u^\top M u$, there exists a $\tilde{v} \in \mathcal{N}$ such that \( | \supp(v) \cup \supp(\tilde{v}) | \leq k \) and $\| v - \tilde{v}\|_2 \leq 1/4$. We have
\begin{align*}
    \| M\|_{k, \mathrm{op}} &= v^\top M v = v^\top M (v - \tilde{v}) + (v - \tilde{v})^\top M \tilde{v} +  \tilde{v}^\top M \tilde{v} \\
    & \le 2 \|v - \tilde{v}\|_2 \| M\|_{k, \mathrm{op}} + \tilde{v}^\top M \tilde{v} %\\ &
    \leq \frac{1}{2}\|M\|_{k,\mathrm{op}} + \sup_{u\in\mathcal{N}} u^\top M u.
\end{align*}
The desired inequality is obtained after rearranging terms in the above display.
\end{proof}

The following lemma describes the asymptotic limit of the nuclear and Frobenius norms of the product of a matrix-variate Beta-distributed random matrix and its reflection. Recall that for $n_1+n_2>p$, we say that a $p\times p$ random matrix $B$ follows a matrix-variate Beta distribution with parameters $n_1/2$ and $n_2/2$, written $B\sim \mathrm{Beta}_p(n_1/2,n_2/2)$, if $B = (S_1+S_2)^{-1/2} S_1 (S_1+S_2)^{-1/2}$, where $S_1\sim W_p(n_1,I_p)$ and $S_2\sim W_p(n_2,I_p)$ are independent Wishart matrices  and $(S_1+S_2)^{1/2}$ is the symmetric matrix square root of $S_1+S_2$. Recall also that the spectral distribution function of any $p\times p$ matrix $A$ is defined as $F^A(t) := n^{-1} \sum_{i=1}^p \mathbbm{1}_{\{ \lambda_i^A \le t\}}$, where $\lambda_i^A$s are eigenvalues (counting multiplicities) of the matrix $A$. Further, given a sequence $(A_n)_{n\in\mathbb{N}}$ of matrices, their limiting spectral distribution function $F$ is defined as the weak limit of the $F^{A_n}$, if it exists.
\begin{lemma}
\label{lemma:spectral-limit}
Let $B\sim \mathrm{Beta}_{p}(n_1/2,n_2/2)$ and suppose that $\lambda_1,\ldots,\lambda_p$ are the eigenvalues of $B$.  Define $a = (a_1,\dots, a_p)^\top$, with  $a_j = \lambda_j(1 - \lambda_j)$ for $j\in[p]$. In the asymptotic regime of~\ref{cond:asymp}, we have
\begin{align*}
    \| {a} \|_1/p &\xrightarrow{\mathrm{a.s.}} \kappa_1 ,\\
    \| {a} \|_2/\sqrt{p} &\xrightarrow{\mathrm{a.s.}} \kappa_2,
\end{align*}
where
\[
\kappa_1=\frac{r}{(1+r)^2(1+s)} \quad \text{and} \quad
\kappa_2^2=\frac{r(r+s-rs+r^2s+rs^2)}{(1+r)^4(1+s)^3}.
\]
\end{lemma}
\begin{proof}
We first look at the limiting spectral distribution of $B$. From the asymptotic relations between $n_1,n_2$ and $p$ in \ref{cond:asymp}, we have that
\[
    p/n_1 \to \xi:= \frac{s+sr}{r+sr} \quad \text{and} \quad
    p/n_2 \to \eta:= \frac{s+sr}{1+s}.
\]
 Define the left and right limits
\begin{equation}
    \label{Eq:tlr}
    t_\ell, t_r := \frac{ (\xi + \eta)\eta + \xi\eta(\xi - \eta) \mp 2\xi\eta\sqrt{\xi - \xi \eta + \eta} }{(\xi + \eta)^2}.
\end{equation}
By \citet[Theorem~1.1]{BaiHuPanZhou2015}, almost surely, weak limit $F$ of $F^B$ exists and is of the form
\(
\max\{1 - 1/\xi, 0\}\delta_0 + \max\{1-1/\eta, 0\}\delta_1 + \mu,
\)
where $\delta_0$ and $\delta_1$ are point masses at $0$ and $1$ respectively, and $\mu$ has a density
\[
\frac{(\xi + \eta)\sqrt{(t_r - t) ( t - t_\ell)} }{ 2 \pi \xi \eta t ( 1-t)}\mathbbm{1}_{[t_\ell,t_r]}
\]
with respect to the Lebesgue measure on $\mathbb{R}$.
%is defined by, for any $A\in \mathbb{B}(\mathbb{R})$,
%    \[
%        F(A) =
%        \int_{A\cap (t_\ell, t_r) } \frac{(\xi + \eta)\sqrt{(t_r - t) ( t - t_\ell)} }{ 2 \pi \xi \eta t ( 1-t)}\, dt + \max\{1 - 1/\xi, 0\}\delta_0(A) + \max\{1-1/\eta, 0\}\delta_1(A),
%    \]
%    where $\delta_t(\cdot): A\in \mathbb{B}(\mathbb{R}) \mapsto \mathbbm{1}(t \in A)$ is the Dirac measure for $t \in \mathbb{R}$.
Define $h_1: t \mapsto t ( 1- t)$. By the portmanteau lemma \citep[see, e.g.][Lemma~2.2]{vanderVaart2000}, we have almost surely that
\begin{align*}
\| a \|_1/p = F^{B} h_1 \rightarrow F h_1 &= \frac{\xi + \eta }{2\pi \xi \eta }\int_{t_\ell}^{t_r}  \sqrt{(t_r-t)(t - t_\ell)} dt = \frac{\xi + \eta }{16 \xi \eta } (t_r-t_\ell)^2 \\
        & = \frac{r}{(1+r)^2(1+s)}.
\end{align*}
Similarly, for $h_2: t \mapsto t^2(1-t)^2$, we have almost surely that
\begin{align*}
    \| a \|_{2}^2/p \to F h_2 &= \frac{\xi + \eta}{2 \pi \xi \eta }\int_{t_\ell}^{t_r} t(1-t) \sqrt{(t_r -t) ( t - t_\ell)}dt\\
    & = \frac{\xi + \eta}{256 \xi \eta }(t_r-t_\ell)^2(8 t_\ell - 5 t_\ell^2 + 8 t_r - 6 t_\ell t_r - 5 t_r^2) \\
    &= \frac{r(r+s-rs+r^2s+rs^2)}{(r+1)^4(s+1)^3}.
\end{align*}
Define $\kappa_1 := F h_1$ and $\kappa_2 := (F h_2)^{1/2}$, we arrive at the lemma.
\end{proof}
The following result concerning the QR decomposition of a Gaussian random matrix is probably well-known. However, since we did not find results in this exact form in the existing literature, we have included a proof here for completeness. Recall that for $n\geq p$, the set $\mathbb{O}^{n\times p}$ can be equipped with a uniform probability measure that is invariant under the action of left multiplication by $\mathbb{O}^{n\times n}$ \citep[see, e.g.\ Stiefel manifold in][Section~2.1.4]{Muirhead1982}.
\begin{lemma}
\label{Lemma:GaussianQR}
Suppose $n\geq p$ and $X$ is an $n\times p$ random matrix with independent $N(0,1)$ entries. Write $X = HT$, with $H$ taking values in $\mathbb{O}^{n\times p}$ and $T$ an upper-triangular $p\times p$ matrix with non-negative diagonal entries. This decomposition is almost surely unique. Moreover, $H$ and $T$ are independent, with $H$ uniformly distributed on $\mathbb{O}^{n\times p}$ with respect to the invariant measure and $T = (t_{j,k})_{j,k\in[p]}$ having independent entries satisfying $t_{j,j}^2 \sim \chi^2_{p-j+1}$ and $t_{j,k}\sim N(0,1)$ for $1\leq j<k\leq p$.
\end{lemma}
\begin{proof}
The uniqueness of the QR decomposition follows since $X$ has rank $p$ almost surely. The marginal distribution of $T$ then follows from the Bartlett decomposition of $X^\top X$ \citep[Theorem~3.2.14]{Muirhead1982} and the relationship between the QR decomposition of $X$ and the Cholesky decomposition of $X^\top X$.

For any fixed $Q\in\mathbb{O}^{n\times n}$, we have $QX \stackrel{\mathrm{d}}= X$. Since $\mathbb{O}^{n\times n}$ acts transitively (by left multiplication) on $\mathbb{O}^{n\times p}$, the joint density of $H$ and $T$ must be constant in $H$ for each value of $T$. In particular, we have that $H$ and $T$ are independent, and that $H$ is uniformly distributed on $\mathbb{O}^{n\times p}$ with respect to the translation-invariant measure.
\end{proof}
\begin{comment}
The lemma below provides concentration inequalities for the norm of a multivariate normal random vector with near-identity covariance.
\begin{lemma}
    \label{Lemma:generalized-chi-square}
Let $X\sim N_d(\mu, \Sigma)$. Then with probability at least $1-\delta$, we have
\[
\|X - \mu\|_2^2 \leq \lambda_1(\Sigma)\{d+2\sqrt{d\log(1/\delta)}+2\log(1/\delta)\}.
\]
\end{lemma}
\begin{proof}
We have $\Sigma^{-1/2}(X-\mu)\sim N_d(0, I_d)$. By~\citet[Lemma~1]{LaurentMassart2000}, we have that the following desired event happens with probability at least $1-\delta$,
\begin{align*}
\|X - \mu\|_2^2 \leq \lambda_1(\Sigma) \|\Sigma^{-1/2}(X-\mu)\|_2^2 \leq \lambda_1(\Sigma)\{d+2\sqrt{d\log(1/\delta)}+2\log(1/\delta)\},
\end{align*}
as desired.
\end{proof}
\end{comment}

The following two lemmas control the moment generation functions of (decoupled) quadratic Rademacher chaos random variables with respect to different matrices.
\begin{lemma}
\label{Lemma:Chaos}
Let $\xi = (\xi_1,\ldots,\xi_d)^\top$ and $\xi'=(\xi'_1,\ldots,\xi'_d)^\top$ be independent with independent Rademacher entries and fix $A\in\mathbb{R}^{d\times d}$.  There exists a universal constant $C>0$ such that for any $0 < \|A\|_{\mathrm{op}} \leq 1/32$, we have
\[
\mathbb{E} (e^{\xi^\top A \xi'})\leq 1 + C\|A\|_{\mathrm{F}}e^{4\|A\|_{\mathrm{F}}^2}.
\]
\end{lemma}
\begin{proof}
By Hoeffding's inequality, we have
\begin{equation}
\label{Eq:CondHoeff}
\mathbb{P}(\xi^\top A \xi' \geq t\mid \xi') \leq \exp\biggl\{-\frac{t^2}{2\|A\xi'\|_2^2}\biggr\}.
\end{equation}
By Jensen's inequality, we have $\mathbb{E}(\|A\xi'\|_2)\leq \{\mathbb{E}(\xi'^\top A^\top A\xi')\}^{1/2} \leq \|A\|_{\mathrm{F}}$. Moreover, the map $x\mapsto \|Ax\|$ is Lipschitz with constant $\|A\|_{\mathrm{op}}$.  Hence, from \citet[Theorem~6.10]{BLM2013}, we have
\begin{equation}
\label{Eq:McDiarmid}
\mathbb{P}(\|A\xi'\|_2 \geq \|A\|_{\mathrm{F}}+u) \leq \exp\biggl\{-\frac{u^2}{8\|A\|_{\mathrm{op}}^2}\biggr\}.
\end{equation}
Combining~\eqref{Eq:CondHoeff} and~\eqref{Eq:McDiarmid}, and setting $u=(2t\|A\|_{\mathrm{op}})^{1/2}$, we have
\begin{align*}
\mathbb{P}(\xi^\top A \xi' \geq t) &\leq \mathbb{P}(\|A\xi'\|_2 \geq \|A\|_{\mathrm{F}}+u) \\
&\qquad  + \mathbb{E}[\mathbb{P}(\xi^\top A \xi' \geq t\mid \xi')\mathbbm{1}_{\{\|A\xi'\|_2\leq \|A\|_{\mathrm{F}}+u\}}]\\
&\leq \exp\biggl\{-\frac{u^2}{8\|A\|_{\mathrm{op}}^2}\biggr\}+\exp\biggl\{-\frac{t^2}{2(\|A\|_{\mathrm{F}}+u)^2}\biggr\} \\
&\leq 2\exp\biggl\{-\frac{t^2}{4(\|A\|_{\mathrm{F}}+t^{1/2}\|A\|_{\mathrm{op}}^{1/2})^2}\biggr\}\\
&\leq 2 \max\bigl\{e^{-t^2/(16\|A\|_{\mathrm{F}}^2)},\, e^{-t/(16\|A\|_{\mathrm{op}})}\bigr\}.
\end{align*}
Consequently, if $32\lambda\|A\|_{\mathrm{op}}\leq 1$, we have
\begin{align*}
    \mathbb{E}(e^{\xi^\top A \xi'}) &= \int_{u=0}^1 \mathbb{P}(e^{\xi^\top A\xi'}\geq u)\,du  + \int_{t=0}^\infty \mathbb{P}(\xi^\top A \xi' \geq t) e^{t} \,dt\\
&\leq 1 + 2\int_{t=0}^{\|A\|_{\mathrm{F}}^2/\|A\|_{\mathrm{op}}} e^{-t^2/(16\|A\|_{\mathrm{F}}^2)+ t}\,dt \\
& \qquad\quad + 2\int_{\|A\|_{\mathrm{F}}^2/\|A\|_{\mathrm{op}}}^\infty  e^{-t/(16\|A\|_{\mathrm{op}})+ t}\,dt\\
& \leq 1 + 8\sqrt{2\pi} \|A\|_{\mathrm{F}}e^{4\| A\|_\mathrm{F}^2} + 64 \|A\|_\mathrm{op}.
\end{align*}
Our claim follows since $\|A\|_{\mathrm{op}}\leq \|A\|_{\mathrm{F}}$.
\end{proof}

\begin{lemma}
\label{Lemma:Chaos2}
Fix $A\in\mathrm{RowSp}(D)\subseteq \mathbb{R}^{p\times p}$ and let $J$ and $J'$ be independent and drawn uniformly at random from all subset of cardinality $k$ of $[p]$. Let $\xi = (\xi_1,\ldots,\xi_d)^\top$ and $\xi'=(\xi_1',\ldots,\xi_d')^\top$ be independent (and independent of $J$ and $J'$) with independent Rademacher entries. Then
\[
\mathbb{E}(e^{\xi^\top A_{J,J'}\xi'}) \leq \biggl\{1 + \frac{Dk}{p}\bigl(\cosh(D\|A\|_{\max}) - 1\bigr)\biggr\}^k.
\]
\end{lemma}
\begin{proof}
Write $a:=\|A\|_{\max}$. Also, for notational simplicity, we define $\theta, \theta' \in\mathbb{R}^p$ such that $\theta_{J}=\xi$, $\theta_{J^\mathrm{c}} = 0$, $\theta'_{J'}=\xi'$, $\theta'_{J'^\mathrm{c}} = 0$. So $\xi^\top A_{J,J'}\xi' = \theta^\top A\theta'$.

For each $j\in [p]$, we write $\mathrm{nb}(j) := \{j'\in[p]: A_{j,j'}\neq 0\}$. Note that by the definition of $\mathrm{RowSp}(D)$, $|\mathrm{nb}(j)|\leq D$ for all $j\in[p]$. Hence,
\[
\theta^\top A \theta'= \sum_{j\in J}\sum_{j'\in \mathrm{nb}(j)\cap J'} A_{j,j'}\theta_j \theta'_{j'} = \sum_{j\in J} c_j \theta_j
\]
where $c_j := \sum_{j'\in\mathrm{nb}(j)\cap J'} A_{j,j'}\theta'_{j'}$. We note that $|c_j|\leq Da$ and $c_j = 0$ unless $j\in\cup_{j'\in J'}\mathrm{nb}(j')$. Observe that $|\cup_{j'\in J'}\mathrm{nb}(j')| \leq Dk$, so $|\cup_{j'\in J'}\mathrm{nb}(j')\cap J|$ is stochastically dominated by the hypergeometric random variable $\mathrm{HyperGeom}(k; Dk, p)$ (defined as the number of black balls obtained from $k$ draws without replacement from an urn containing $p$ balls, $Dk$ of which are black). Let $B\sim \mathrm{Bin}(k, Dk/p)$, we have by \citet[Theorem~4]{Hoeffding1963} that
\begin{align*}
    \mathbb{E}(e^{\theta^\top A\theta'}) &= \mathbb{E}\bigl\{\prod_{j\in[p]} \mathbb{E}(e^{c_j\theta_j}\mid J', \theta')\bigr\} = \mathbb{E}\bigl\{\prod_{j\in \cup_{j'\in J'}\mathrm{nb}(j')\cap J} \cosh(c_j)\bigr\} \\
    &\leq  \mathbb{E} (e^{B\log\cosh(Da)}) = \biggl\{1 + \frac{Dk}{p}\bigl(\cosh(Da) - 1\bigr)\biggr\}^k.
\end{align*}
The proof is complete.
\end{proof}
%For the second claim, we apply Hanson--Wright inequality for Rademacher random variables given by \citet[Theorem~1.2]{Talagrand1996} to obtain that for some universal constant $c>0$, we have
%\[
%\mathbb{P}\bigl(|\xi^\top A \xi| \geq \mathrm{tr}(A) + t\bigr) \leq 2\max\bigl\{e^{-ct^2/\|A\|_{\mathrm{F}}^2}, e^{-ct/\|A\|_{\mathrm{op}}}\bigr\}
%\]
%Consequently, if $2\lambda\|A\|_{\mathrm{op}} \leq c$, we have
%\begin{align*}
%\mathbb{E}(e^{\lambda \xi^\top A \xi}) &= \int_{u=0}^\infty \mathbb{P}(e^{\lambda \xi^\top A\xi}\geq u)\,du \leq e^{\lambda\, \mathrm{tr}(A)} + \int_{t=0}^\infty \mathbb{P}(\xi^\top A \xi \geq t) e^{\lambda t}\,dt\\
%&\leq e^{\lambda\, \mathrm{tr}(A)}  + 2\int_{t=0}^{\|A\|_{\mathrm{F}}^2/\|A\|_{\mathrm{op}}} e^{-ct^2/\|A\|_{\mathrm{F}}^2+\lambda t}\,dt + 2\int_{\|A\|_{\mathrm{F}}^2/\|A\|_{\mathrm{op}}}^\infty  e^{-ct/\|A\|_{\mathrm{op}}+\lambda t}\,dt\\
%&\leq e^{\lambda\, \mathrm{tr}(A)} +2\sqrt{\pi/c}\|A\|_{\mathrm{F}} e^{\lambda^2\|A\|_{\mathrm{F}}^2/(4c)} + 4\|A\|_{\mathrm{op}}/c,
%\end{align*}
%as desired.

% \input{contents/acknowledgement.tex}

%%%%%%%%%%%%%%%%%%%%%%%%%%%%%%%%%%%%%%%%%%%%%%%%%%%%
%%%%%%%%%%%%%%%%%%%%%%%%%%%%%%%%%%%%%%%%%%%%%%%%%%%%
\bibliographystyle{custom}
\bibliography{compsket.bib}

\begin{thebibliography}{42}
\providecommand{\natexlab}[1]{#1}
\providecommand{\url}[1]{\texttt{#1}}
\providecommand{\urlprefix}{URL }
\providecommand{\eprint}[2][]{\url{#2}}

\bibitem[{Arias-Castro, Cand{\`e}s and Plan(2011)}]{AriasCastroetal2011}
Arias-Castro, E., Cand{\`e}s, E.~J. and Plan, Y. (2011) Global testing under
  sparse alternatives: ANOVA, multiple comparisons and the higher criticism.
  \emph{Ann. Statist.}, \textbf{39}, 2533--2556.

\bibitem[{Bai et~al.(2020)Bai, Zhang, Fu, Chen, Zhang, Huang, Li, Li and
  Wu}]{bai2020targeting}
Bai, F., Zhang, P., Fu, Y., Chen, H., Zhang, M., Huang, Q., Li, D., Li, B. and
  Wu, K. (2020) Targeting ANXA1 abrogates Treg-mediated immune suppression in
  triple-negative breast cancer. \emph{Journal for immunotherapy of cancer},
  \textbf{8}.

\bibitem[{Bai et~al.(2015)Bai, Hu, Pan and Zhou}]{BaiHuPanZhou2015}
Bai, Z., Hu, J., Pan, G. and Zhou, W. (2015) Convergence of the empirical
  spectral distribution function of Beta matrices. \emph{Bernoulli},
  \textbf{21}, 1538--1574.

\bibitem[{Bayer, Yu and Malek(2007)}]{bayer2007function}
Bayer, A.~L., Yu, A. and Malek, T.~R. (2007) Function of the IL-2R for thymic
  and peripheral CD4+ CD25+ Foxp3+ T regulatory cells. \emph{The Journal of
  Immunology}, \textbf{178}, 4062--4071.

\bibitem[{Birg\'e(2001)}]{Birge2001}
Birg\'e, L. (2001) An alternative point of view on Lepski’s method. In de
  Gunst, M and Klaassen, C. and van der Vaart, A. Eds,. \emph{Lecture
  Notes-Monograph Series 36}, 113--133.

\bibitem[{Boucheron, Lugosi and Massart(2013)}]{BLM2013}
Boucheron, S., Lugosi, G. and Massart, P. (2013) \emph{Concentration
  inequalities: A nonasymptotic theory of independence}. Oxford university
  press.

\bibitem[{Cai, Liu and Xia(2014)}]{cai2014two}
Cai, T.~T., Liu, W. and Xia, Y. (2014) Two-sample test of high dimensional
  means under dependence. \emph{J. Roy. Statist. Soc., Ser. B}, \textbf{76},
  349--372.

\bibitem[{Carpentier et~al.(2019)Carpentier, Collier, Comminges, Tsybakov and
  Wang}]{carpentier2019minimax}
Carpentier, A., Collier, O., Comminges, L., Tsybakov, A.~B. and Wang, Y. (2019)
  Minimax rate of testing in sparse linear regression. \emph{Automation and
  Remote Control}, \textbf{80}, 1817--1834.

\bibitem[{Carpentier and Verzelen(2021)}]{carpentier2021optimal}
Carpentier, A. and Verzelen, N. (2021) Optimal sparsity testing in linear
  regression model. \emph{Bernoulli}, \textbf{27}, 727--750.

\bibitem[{Charbonnier, Verzelen and Villers(2015)}]{Chabonnier2015}
Charbonnier, C., Verzelen, N. and Villers, F. (2015) A global homogeneity test
  for high-dimensional linear regression. \emph{Electron. J. Statist.},
  \textbf{9}, 318--382.

\bibitem[{Chen, Li and Zhong(2019)}]{chen2019two}
Chen, S.~X., Li, J. and Zhong, P.-S. (2019) Two-sample and ANOVA tests for high
  dimensional means. \emph{Ann. Statist.}, \textbf{47}, 1443--1474.

\bibitem[{Chow(1960)}]{chow1960tests}
Chow, G.~C. (1960) Tests of equality between sets of coefficients in two linear
  regressions. \emph{Econometrica}, \textbf{28}, 591--605.

\bibitem[{Dicker(2014)}]{Dicker2014}
Dicker, L.~H. (2014) Variance estimation in high-dimensional linear models.
  \emph{Biometrika}, \textbf{101}, 269--284.

\bibitem[{Doebbeler et~al.(2018)Doebbeler, Koenig, Krzyzak, Seitz, Wild, Ulas,
  Ba{\ss}ler, Kopelyanskiy, Butterhof, Kuhnt et~al.}]{doebbeler2018cd83}
Doebbeler, M., Koenig, C., Krzyzak, L., Seitz, C., Wild, A., Ulas, T.,
  Ba{\ss}ler, K., Kopelyanskiy, D., Butterhof, A., Kuhnt, C. et~al. (2018) CD83
  expression is essential for Treg cell differentiation and stability.
  \emph{JCI insight}, \textbf{3}.

\bibitem[{Donoho and Jin(2004)}]{donoho2004higher}
Donoho, D. and Jin, J. (2004) Higher criticism for detecting sparse
  heterogeneous mixtures. \emph{Ann. Statist.}, \textbf{32}, 962--994.

\bibitem[{Efron et~al.(2004)Efron, Hastie, Johnstone,  and
  Tibshirani}]{Efronetal2004}
Efron, B., Hastie, T., Johnstone, I.,  and Tibshirani, R. (2004) Least angle
  regression. \emph{Ann. Statist.}, \textbf{32}, 407--499.

\bibitem[{Fan, Guo and Hao(2012)}]{fan2012variance}
Fan, J., Guo, S. and Hao, N. (2012) Variance estimation using refitted
  cross-validation in ultrahigh dimensional regression. \emph{J. Roy. Statist.
  Soc., Ser. B}, \textbf{74}, 37--65.

\bibitem[{Hoeffding(1963)}]{Hoeffding1963}
Hoeffding, W. (1963) Probability inequalities for sums of bounded random
  variables. \emph{J. Amer. Statist. Assoc.}, \textbf{58}, 13--30.

\bibitem[{Homrighausen and McDonald(2013)}]{homrighausen2013lasso}
Homrighausen, D. and McDonald, D. (2013) The lasso, persistence, and
  cross-validation. In \emph{International Conference on Machine Learning},
  1031--1039, PMLR.

\bibitem[{Ingster(1997)}]{Ingster1997}
Ingster, Y.~I. (1997) Some problems of hypothesis testing leading to infinitely
  divisible distribution. \emph{Math. Methods Statist.}, \textbf{6}, 47--69.

\bibitem[{Ingster, Tsybakov and Verzelen(2010)}]{ITV10}
Ingster, Y.~I., Tsybakov, A. and Verzelen, N. (2010) Detection boundary in
  sparse regression. \emph{Electron. J. Statist.}, \textbf{4}, 1476--1526.

\bibitem[{Kannel and McGee(1979)}]{kannel1979diabetes}
Kannel, W.~B. and McGee, D.~L. (1979) Diabetes and cardiovascular disease: the
  Framingham study. \emph{J. Amer. Medical Assoc.}, \textbf{241}, 2035--2038.

\bibitem[{Kim et~al.(2015)Kim, Barnitz, Kreslavsky, Brown, Moffett, Lemieux,
  Kaygusuz, Meissner, Holderried, Chan et~al.}]{kim2015stable}
Kim, H.-J., Barnitz, R.~A., Kreslavsky, T., Brown, F.~D., Moffett, H., Lemieux,
  M.~E., Kaygusuz, Y., Meissner, T., Holderried, T.~A., Chan, S. et~al. (2015)
  Stable inhibitory activity of regulatory T cells requires the transcription
  factor Helios. \emph{Science}, \textbf{350}, 334--339.

\bibitem[{Kraft and Hunter(2009)}]{kraft2009genetic}
Kraft, P. and Hunter, D.~J. (2009) Genetic risk prediction---are we there yet?
  \emph{N. Engl. J. Medcine}, \textbf{360}, 1701--1703.

\bibitem[{Laurent and Massart(2000)}]{LaurentMassart2000}
Laurent, B. and Massart, P. (2000) Adaptive estimation of a quadratic
  functional by model selection. \emph{Ann. Statist.}, \textbf{28}, 1302--1338.

\bibitem[{Mahoney(2011)}]{mahoney2011randomized}
Mahoney, M.~W. (2011) Randomized algorithms for matrices and data. \emph{Found.
  Trends Mach. Learn.}, \textbf{3}, 123--224.

\bibitem[{Muirhead(2009)}]{Muirhead1982}
Muirhead, R.~J. (2009) \emph{Aspects of Multivariate Statistical Theory}. John
  Wiley \& Sons, Hoboken, New Jersey.

\bibitem[{Reid, Tibshirani and Friedman(2016)}]{reid2016study}
Reid, S., Tibshirani, R. and Friedman, J. (2016) A study of error variance
  estimation in lasso regression. \emph{Statist. Sinica}, \textbf{26}, 35--67.

\bibitem[{St{\"a}dler and Mukherjee(2012)}]{stadler2012two}
St{\"a}dler, N. and Mukherjee, S. (2012) Two-sample testing in high-dimensional
  models. \emph{arXiv preprint}, arxiv:1210.4584.

\bibitem[{Sun and Zhang(2012)}]{sun2012scaled}
Sun, T. and Zhang, C.-H. (2012) Scaled sparse linear regression.
  \emph{Biometrika}, \textbf{99}, 879--898.

\bibitem[{Suo et~al.(2022)Suo, Dann, Goh, Jardine, Kleshchevnikov, Park,
  Botting, Stephenson, Engelbert, Tuong, Polanski, Yayon, Xu, Suchanek,
  Elmentaite, Conde, He, Pritchard, Miah, Moldovan, Steemers, Prete, Marioni,
  Clatworthy, Haniffa and Teichmann}]{Suo2022.01.17.476665}
Suo, C., Dann, E., Goh, I., Jardine, L., Kleshchevnikov, V., Park, J.-E.,
  Botting, R.~A., Stephenson, E., Engelbert, J., Tuong, Z.~K., Polanski, K.,
  Yayon, N., Xu, C., Suchanek, O., Elmentaite, R., Conde, C.~D., He, P.,
  Pritchard, S., Miah, M., Moldovan, C., Steemers, A.~S., Prete, M., Marioni,
  J.~C., Clatworthy, M.~R., Haniffa, M. and Teichmann, S.~A. (2022) Mapping the
  developing human immune system across organs. \emph{bioRxiv preprint},
  doi.org/10.1101/2022.01.17.476665.

\bibitem[{Tibshirani(1996)}]{tibshirani1996regression}
Tibshirani, R. (1996) Regression shrinkage and selection via the lasso.
  \emph{Journal of the Royal Statistical Society: Series B (Methodological)},
  \textbf{58}, 267--288.

\bibitem[{Toomer et~al.(2019)Toomer, Lui, Altman, Ban, Chen and
  Malek}]{toomer2019essential}
Toomer, K.~H., Lui, J.~B., Altman, N.~H., Ban, Y., Chen, X. and Malek, T.~R.
  (2019) Essential and non-overlapping IL-2R$\alpha$-dependent processes for
  thymic development and peripheral homeostasis of regulatory T cells.
  \emph{Nature communications}, \textbf{10}, 1--16.

\bibitem[{van~der Vaart(2000)}]{vanderVaart2000}
van~der Vaart, A.~W. (2000) \emph{Asymptotic Statistics}. Cambridge University
  Press, Cambridge.

\bibitem[{Vershynin(2012)}]{Vershynin2012}
Vershynin, R. (2012) Introduction to the non-asymptotic analysis of random
  matrices. In Y. Eldar and G. Kutyniok (Eds.). \emph{Compressed Sensing,
  Theory and Applications}, 210--268.

\bibitem[{Wainwright(2019)}]{Wainwright2019}
Wainwright, M.~J. (2019) \emph{High-Dimensional Statistics: A Non-Asymptotic
  Viewpoint}. Cambridge University Press, Cambridge.

\bibitem[{Walker(2013)}]{walker2013treg}
Walker, L.~S. (2013) Treg and CTLA-4: two intertwining pathways to immune
  tolerance. \emph{Journal of autoimmunity}, \textbf{45}, 49--57.

\bibitem[{Wang, Berthet and Samworth(2016)}]{wang2016statistical}
Wang, T., Berthet, Q. and Samworth, R.~J. (2016) Statistical and computational
  trade-offs in estimation of sparse principal components. \emph{Ann.
  Statist.}, \textbf{44}, 1896--1930.

\bibitem[{Xia, Cai and Cai(2015)}]{xia2015testing}
Xia, Y., Cai, T. and Cai, T.~T. (2015) Testing differential networks with
  applications to the detection of gene-gene interactions. \emph{Biometrika},
  \textbf{102}, 247--266.

\bibitem[{Xia, Cai and Cai(2018)}]{xia2018two}
Xia, Y., Cai, T. and Cai, T.~T. (2018) Two-sample tests for high-dimensional
  linear regression with an application to detecting interactions.
  \emph{Statist. Sinica}, \textbf{28}, 63--92.

\bibitem[{Xia, Cai and Sun(2020)}]{xia2020gap}
Xia, Y., Cai, T.~T. and Sun, W. (2020) Gap: A general framework for information
  pooling in two-sample sparse inference. \emph{J. Amer. Statist. Assoc.},
  \textbf{115}, 1236--1250.

\bibitem[{Zhu and Bradic(2016)}]{ZhuBradic2016}
Zhu, Y. and Bradic, J. (2016) Two-sample testing in non-sparse high-dimensional
  linear models. \emph{arXiv preprint}, arxiv:1610.04580.

\end{thebibliography}
\end{document}